\documentclass[12pt]{article}
\usepackage[margin=1in]{geometry}
\usepackage[round]{natbib}
\usepackage{amsmath,amsfonts,amsthm,graphicx}

\newtheorem{theorem}{Theorem}
\newtheorem{lemma}{Lemma}
\newtheorem{corollary}{Corollary}
\newtheorem{proposition}{Proposition}
\theoremstyle{definition}

\title{A nonparametric regression alternative to empirical Bayes approaches to simultaneous estimation}
\author{Alton Barbehenn and Sihai Dave Zhao}

\begin{document}
\maketitle

\begin{abstract}

  The simultaneous estimation of multiple unknown parameters lies at heart of a broad class of important problems across science and technology. Currently, the state-of-the-art performance in the such problems is achieved by nonparametric empirical Bayes methods. However, these approaches still suffer from two major issues. First, they solve a frequentist problem but do so by following Bayesian reasoning, posing a philosophical dilemma that has contributed to somewhat uneasy attitudes toward empirical Bayes methodology. Second, their computation relies on certain density estimates that become extremely unreliable in some complex simultaneous estimation problems. In this paper, we study these issues in the context of the canonical Gaussian sequence problem. We propose an entirely frequentist alternative to nonparametric empirical Bayes methods by establishing a connection between simultaneous estimation and penalized nonparametric regression. We use flexible regularization strategies, such as shape constraints, to derive accurate estimators without appealing to Bayesian arguments. We prove that our estimators achieve asymptotically optimal regret and show that they are competitive with or can outperform nonparametric empirical Bayes methods in simulations and an analysis of spatially resolved gene expression data.
\end{abstract}

\section{ Introduction }
\label{section-intro}

We consider the problem of simultaneously estimating multiple unknown parameters of interest. A canonical problem is to estimate a vector $(\theta_1, \dots, \theta_n)^\top$ given independent observations $X_i \sim N(\theta_i, \sigma^2)$ for $i = 1, \ldots, n$, where $\sigma > 0$ is known. This Gaussian sequence problem \citep{Johnstone2019} is a foundational model in theoretical statistics because it captures the core issues at the heart of a broad class of important problems across science and technology. For example, in large-scale A/B testing, a large number of randomized experiments are conducted to evaluate the effects of multiple types of treatments on multiple types of metrics, and a common goal is to simultaneously estimate every treatment effect \citep{guo2020empirical, Ignatiadis2021, muralidharan2010empirical}. In genomics it is common to profile tens of thousands of genomic features, and a frequent goal is to simultaneously estimate the association between each feature and some phenotype \citep{love2014moderated, robinson2010edger, smyth2004linear}. Other domains where similar issues arise include small area estimation \citep{rao2015small}, econometrics \citep{koenker2023empirical}, and large-scale inference \citep{efron2012large}. As the scale, complexity, and number of unknown parameters in moden experiments has exploded, understanding simultaneous estimation has become increasingly important.

The key statistical feature of simultaneous estimation problems is that the most natural approach, where each parameter is estimated separately, can be suboptimal. In the Gaussian sequence problem, this corresponds to estimating each $\theta_i$ using $X_i$, and early results \citep{Robbins1951, Stein1956} found that this estimator is inadmissible under squared error loss when $n \geq 3$. The overall lesson of this and subsequent work is that estimating $\theta_i$ by including information from the other indices $j \neq i$, even though they may seem unrelated, can improve overall accuracy \citep{Fourdrinier2018, Johnstone2019}. These estimators are often called shrinkage estimators, because early estimators such as one by \citet{JamesStein1961} dominated the naive estimators $X_i$ by shrinking them toward zero. These discoveries came as a major surprise and set off a long and rich line of research to better understand shrinkage and to develop more accurate estimators. Our goal in this paper is to further enrich our understanding of this phenomenon.

The modern approach to simultaneous estimation is to use empirical Bayes methods \citep{Efron2014fg, Efron2019, zhang2003compound}. These methods are Bayesian in the sense that they pretend that the unknown parameters are random samples from some prior distribution, then estimate the parameters using posterior means or medians. They are ``empirical'' Bayesian because the prior is estimated using the observed data, since it is generally unknown. In the Gaussian sequence problem, the empirical Bayes approach pretends that the true means $\theta_i$ are random with some prior distribution $G$ and estimates $\theta_i$ using its posterior mean
\begin{equation}
\label{Tweedie}
E(\theta \mid X = X_i) 
    = \frac{\int_{-\infty}^{\infty} t \phi_\sigma(t-X_i) dG(t)}{\int_{-\infty}^{\infty} \phi_\sigma(t-X_i) dG(t)}
    = X_i + \sigma^2 \frac{f_G'(X_i)}{f_G(X_i)},
  \end{equation}
where $\phi_\sigma(\cdot)$ denotes the density function of a $N(0, \sigma^2)$ random variable and $f_G(x)$ denotes what the marginal density of the $X_i$ would be if the $\theta_i$ indeed were random. The second equality is often called Tweedie's formula \citep{Tweedie}. Since $G$ is unknown, \eqref{Tweedie} cannot be calculated, so the observed data are used to construct plug-in estimates of $E(\theta \mid X = x)$. \citet{Efron2014fg} divides empirical Bayes methods into two classes: $f$-modeling and $g$-modeling, depending on whether the marginal $f_G$ or prior $G$ is substituted into \eqref{Tweedie}, respectively. For example, \citet{EfronMorris1973} showed that the shrinkage estimator of \citet{JamesStein1961} can be recovered as a $g$-modeling procedure under a normal assumption on the prior.

Currently, the state-of-the-art performance in the Gaussian sequence problem is achieved by nonparametric empirical Bayes methods. \cite{BrownGreenshtein2009} proposed an $f$-modeling procedure that estimates $f_G$ using a kernel density estimator and \citet{JiangZhang2009} proposed a $g$-modeling procedure that estimates $G$ using nonparametric maximum likelihood \citep{KieferWolfowitz1956},
which \citet{KoenkerMizera2014} showed can be efficiently calculated using modern convex optimization techniques. These methods are asymptotically optimal under squared-error risk, in the following sense. \citet{Robbins1951} formulated the Gaussian sequence problem as finding a decision rule $\delta(X_1, \ldots, X_n)$ that has a small mean squared error $E \sum_i \{ \theta_i - \delta_i(X_1, \ldots, X_n) \}^2$. He referred to this as an example of a compound decision problem, as multiple decisions need to be made to estimate multiple $\theta_i$. He focused on the class of so-called separable estimators, where $\delta_i(X_1, \ldots, X_n) = d(X_i)$ for some fixed function $d$, and showed that
\begin{equation}
\label{oracle}
d^*(x) = \frac{\sum_{i=1}^n \theta_i \phi_\sigma(\theta_i-x)}{\sum_{i=1}^n \phi_\sigma(\theta_i-x)}.
\end{equation}
minimizes mean squared error over all separable estimators. The estimator $\hat{d}_{JZ}(x)$ of \citet{JiangZhang2009} is asymptotically optimal in that
\begin{equation}
  \label{ratio_opt}
  \limsup_{n \rightarrow \infty} \frac{
    E \sum_i [\{\theta_i - \hat{d}_{JZ}(X_i)\}^2]
  }{
    E \sum_i [\{\theta_i - d^*(X_i)\}^2]
  }
  =
  1
\end{equation}
for any $(\theta_1, \ldots, \theta_n)^\top$ under mild conditions, and the estimator $\hat{d}_{BG}(x)$ of \citet{BrownGreenshtein2009} enjoys a similar property.

However, despite their excellent performance, nonparametric empirical Bayes approaches still suffer from two major issues. The first is philosophical: these methods solve a frequentist problem, but they do so by following Bayesian reasoning. They have good frequentist properties, but this ``philosophical identity problem'' has nevertheless contributed to somewhat uneasy attitudes toward empirical Bayesian methods that may limit their uptake \citep{Efron2019}. The second issue is practical and emerges in recent work that has started to consider how nonparametric empirical Bayes approaches can be applied to more complex problems. These include estimating multidimensional $\theta_i$ \citep{saha2020on, soloff2021multivariate}, covariance matrix estimation \citep{xin2022compound}, and regression \citep{kim2022flexible, wang2021nonparametric}. For example, \citet{wang2021nonparametric} studied the multivariate regression setting $Y_{ik} = X_i^\top \beta_k^\star + \epsilon_{ik} \sim N(0, \tau_k^2)$ for outcomes $k = 1, \ldots, K$. Their nonparametric empirical Bayes method pretends that $\beta_k^\star \sim G(b)$ and estimates $G(b)$ by maximizing the marginal likelihood:
\[
  \hat{G}(b) = \arg\max_G \prod_{k = 1}^K \int \frac{1}{\{2 \pi (\tau^2 / n)\}^{1/2}} \exp\left\{-\frac{n}{2 \tau^2} (\bar{Y}_k - \bar{X}^\top b)^2 \right\} dG(b),
\]
where $n$ is the sample size, $\bar{Y}_k$ and $\bar{X}$ are the sample means of the $Y_{ik}$ and $X_i$. \citet{KoenkerMizera2014} noted that an approximation to $\hat{G}(b)$ can be conveniently obtained by maximizing over all discrete distributions supported on a pre-specified grid of points, which is a convex optimization problem. However, when $n$ is large, the density in the integrand can be extremely small at most of the grid points, which caused problems even when using modern optimization software such as MOSEK.

In this paper we study these issues using the canonical Gaussian sequence problem. In this setting we model the low density issue described above by generating $X_i \sim N(\theta_i, \sigma^2)$ with low noise. When $\sigma^2$ is small, the densities that arise in these problems can suffer from the same numerical issues as those encountered in the more complex simultaneous estimation problems. This is further discussed in Section \ref{section-low-noise} and may account for the poor performance of $\hat{d}_{JZ}(x)$ and $\hat{d}_{BG}(x)$ in simulations.

To avoid these issues, we propose a new, alternative approach to the Gaussian sequence problem. First, our approach is purely frequentist and is based on the compound decision theoretic formulation of \citet{Robbins1951}. We establish a connection between simultaneous estimation and nonparametric regression and adopt an empirical risk minimization strategy that converts the problem into a penalized least-squares problem. We use flexible regularization strategies, such as shape constraints, to derive accurate estimators without appealing to Bayesian arguments. Second, our new approach can still achieve a similar accuracy as nonparametric empirical Bayes methods. We prove that our estimators achieve asymptotically optimal regret. Finally, our approach is robust to small $\sigma^2$. We show in simulations that the performances of our estimators are comparable to those of $\hat{d}_{BG}(x)$ and $\hat{d}_{JZ}(x)$ in high-noise settings and superior in low-noise settings. We also apply these methods to denoise spatially resolved gene expression data and show that our proposed procedures achieve the best results. In the future, we hope that generalizations of our approach can lead to novel estimators with distinct advantages in more complex simultaneous estimation problems.

\section{ A Penalized Regression Framework }
\label{section-erm}

Our framework is motivated by the observation by \citet{Stigler1990} that finding a good estimator of the oracle separable estimator $d^*(x)$ \eqref{oracle} closely resembles least-squares regression. Specifically, if we had the oracle coordinate pairs $(X_i, \theta_i)$, we could model and estimate $d^*(x)$ by minimizing the oracle loss
\begin{equation}
\label{oracle-loss}
R_n(d) = n^{-1} \sum_{i=1}^n \{\theta_i - d(X_i)\}^2 
\end{equation}
using ordinary least squares. Minimizing $R_n(d)$ is not possible because $\theta_i$ are unknown. \cite{Stigler1990} solved this problem for the class of linear models $d_{a,b}(x) = a + b x$. He first found the oracle linear model assuming $\theta_i$ were known and then developed method-of-moments-type estimates of those coefficients. Interesting, this essentially gave Stein-type shrinkage estimators and provided a new interpretation of shrinkage. However, his method-of-moments approach is restricted to the linear class $d_{a, b}(x)$, so this strategy cannot give estimators as accurate as the nonlinear ones produced by nonparametric empirical Bayes methods.

Our main insight is to combine Stigler's regression perspective with empirical risk minimization ideas. Instead of deriving the oracle minimizer of $R_n(d)$ and then estimating it, as in \cite{Stigler1990}, our approach is to estimate $R_n(d)$ and then minimize the empirical loss. This strategy allows us to take advantage of many useful tools developed in the nonparametric regression and empirical risk minimization literatures, such as the easy implementation of shape constraints and a well-established set of techniques for theoretical analysis \citep{VandervaartWellner1996, Wainwright2019}. 

Our main task is therefore to develop a computable estimate of $R_n(d)$ that can be directly minimized. A natural approach is to plug in $X_i$ for each $\theta_i$ to produce the loss estimate $\hat{R}_{\text{naive}}(d) = n^{-1} \sum_{i=1}^n \{X_i - d(X_i)\}^2$, but without further regularization $\hat{R}_{\text{naive}}(d)$ can be exactly minimized by the identity function $d(x) = x$. This is a problem because the identity function corresponds to the naive estimator, so we know that the minimizer of $\hat{R}_{\text{naive}}(d)$ will never produce asymptotically optimal estimators.

Instead, we use Stein's lemma \citep{SteinsLemma} to develop an unbiased estimate of $R_n(d)$ for a large class of $d(x)$. If $d(x)$ is an absolutely continuous function, then
\begin{align}
R(d;\theta)
    &= \frac{1}{n} \sum_{i=1}^n E \{\theta_i - d(X_i)\}^2 \nonumber \\
    &= \frac{1}{n} \sum_{i=1}^n E \{X_i - d(X_i)\}^2 + \frac{2}{n} \sum_{i=1}^n E \{ (X_i - \theta_i) d(X_i) \} - \sigma^2 \nonumber \\
    &= \frac{1}{n} \sum_{i=1}^n E \{X_i - d(X_i)\}^2 + \frac{2 \sigma^2}{n} \sum_{i=1}^n E \{ d'(X_i) \} - \sigma^2 . \label{risk-sum}
\end{align}
We note that it is actually sufficient for $d(x)$ to be weakly differentiable; however, we will stick to absolute continuity for consistency with our later discussion. With this risk decomposition, an unbiased estimate of $R_n(d)$ is: 
\begin{equation}
\label{SURE}
\hat{R}_1(d) = \frac{1}{n} \sum_{i=1}^n \{X_i - d(X_i)\}^2 + \frac{2 \sigma^2}{n} \sum_{i=1}^n d'(X_i) - \sigma^2 . 
\end{equation}
In our regresson framework, the risk estimate \eqref{SURE}, often called Stein's unbiased risk estimate (SURE), is readily interpreted as a penalized least-squares loss. 
We note that in contrast to most penalized least-squares regression problems, $\hat{R}_1(d)$ does not have a tuning parameter that governs the trade-off between the goodness-of-fit and parsimony of the fitted function. 

The penalized loss \eqref{SURE} has interesting connections to existing methodology. First, since Stein's lemma is used to compute the covariance between $X_i$ and $d(X_i)$, the penalty term can be thought of as a covariance penalty similar to Akaike's information criterion and Mallow's $C_p$ criterion \citep{Efron2014cov, Akaike1973, Mallows1973}. In contrast to most uses of these information criteria, here we use the covariance penalty in the fitting procedure rather than for model evaluation after fitting. Second, if we restrict $d(x)$ to be monotone non-decreasing, the penalty term can be seen as an $\ell_1$-type penalty; this suggests many connections to existing $\ell_1$-penalized least-squared regression problems. For example, if we parameterize $d(x)$ to be a continuous, piecewise linear function with knots at each $X_i$, then \eqref{SURE} becomes an irregularly spaced, degree-one trend-filter \citep{Tibshirani2014}. Third, for any monotone non-decreasing $d(x)$, the penalty can be interpreted as a weighted total variation penalty. To see this, let $F_n(t)$ denote the empirical distribution of $X_i$. Then $n^{-1} \sum_{i=1}^n d'(X_i) = \int_{-\infty}^{\infty} \vert d'(t) \vert dF_n(t)$, which is the total variation with respect to the empirical measure instead of the Lebesgue measure. The usual total variation is used as a penalty for adaptive regression splines \citep{MammenvandeGeer1997} and a constraint for saturating splines \citep{Boyd2018}. Based on this interpretation, our approach can be interpreted as estimating the optimal $d^*(x)$ \eqref{oracle} using a self-supervised, isotonic, $\ell_1$-penalized least-squares regression problem. 

To illustrate the simplest nontrivial application of the penalized regression framework, we return to the class of simple linear estimators used by \cite{Stigler1990}: $d_{a,b}(x) = a + b x$. Since estimators in this class have a constant slope, our estimator is given by:
\begin{equation}
\label{JamesStein}
\hat{d}_{a,b}(x) = \arg\min_{a,b\in\mathbb{R}} \frac{1}{n} \sum_{i=1}^n (X_i - a - b X_i)^2 + 2 \sigma^2 b - \sigma^2 .
\end{equation}
Our resulting estimates are $\hat{a} = (1-\hat{b}) \bar{X}_n$ and $\hat{b} = 1 - n \sigma^2 S^{-2}$, where $S^2 = \sum_{i=1}^n (X_i - \bar{X}_n)^2$ and $\bar{X}_n$ is the sample mean. These estimates are very close to the Efron-Morris estimator \citep{EfronMorris1973} and Stigler's method-of-moments type estimator \citep{Stigler1990}, which use $n-3$ and $n-1$ rather than $n$ in $\hat{b}$, respectively.


Related approaches have been considered for this problem before. Many empirical Bayes methods \citep{BanerjeeLiuMukherjeeSun2021, XieKouBrown2012, XieKouBrown2016, Zhao2021, ZhaoBiscarri2021} use empirical risk estimation to fit various parametric classes of $d(x)$. These methods use Bayesian ideas to motivate the parametric class. In contrast our approach directly targets the risk of interest 
and can thus produce estimators that have no natural Bayesian interpretation. By analogy to Efron's classification of empirical Bayes methods into $f$- and $g$-modeling estimators \citep{Efron2014fg}, we can consider our procedure as ``$e$-modeling'', as we directly model the Bayesian conditional expectation 
$E(\theta \mid X)$. 

\section{Regression via Constrained SURE Minimization}
\label{section-ecdf}

A natural approach to developing an asymptotically optimal estimator based on the discussion above is to minimize  $\hat{R}_1(d)$ \eqref{SURE} over a suitable function class.
However, it turns out that the penalty in $\hat{R}_1(d)$ is not sufficient to prevent overfitting. In other words, there exist absolutely continuous functions such that $d(X_i) = X_i$ but $d'(X_i) \leq 0$ for each $i = 1, \dots, n$; these can simultaneously minimize the squared error while having an arbitrarily small penalty. Figure \ref{fig-riskminimizers}A illustrates two examples, denoted $d_{\text{sawtooth},a}(x)$. These are continuous, piecewise linear functions with two knots evenly spaced between each consecutive pair of observations so that $d_{\text{sawtooth},a}(X_i) = X_i$ and $d_{\text{sawtooth},a}'(X_i) = a$ for every $i = 1, \ldots, n$ and some $a \in \mathbb{R}$. Because $\hat{R}_1(d_{\text{sawtooth},a}) = 2 \sigma^2 a - \sigma^2$, we can make $\hat{R}_1(d_{\text{sawtooth},a})$ arbitrarily small by decreasing $a$, yet $d_{\text{sawtooth},a}(X_i)$ would estimate $\theta_i$ to be $X_i$. Furthermore, notice that $d_{\text{sawtooth},0}(x)$ is monotone non-decreasing, so the addition of a monotonicity constraint will still not prevent overfitting.

\begin{figure}[t!]
  \centering
  \includegraphics[width=0.8\linewidth]{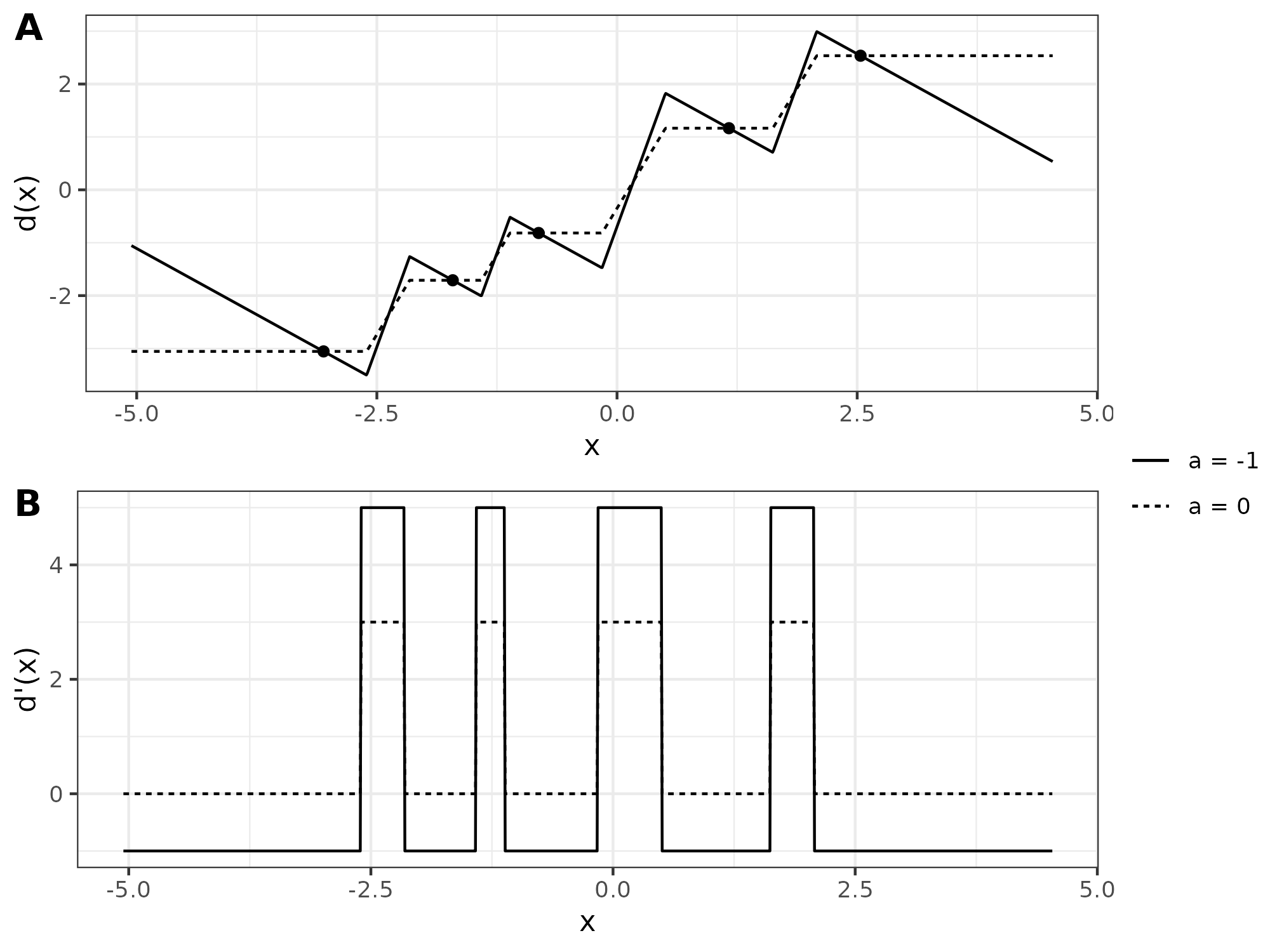}
  \caption{Two examples of absolutely continuous functions that satisfy $d_{\text{sawtooth},a}(X_i) = X_i$ and $d_{\text{sawtooth},a}'(X_i) = a$ for every $i = 1, \ldots, n$ and some $a \in \mathbb{R}$. Panel A plots $d_{\text{sawtooth},a}(x)$ functions and panel B shows their derivatives. The lines are labeled with their slopes, $a$, at each $X_i$. The dashed line is monotone non-decreasing.}
  \label{fig-riskminimizers}
\end{figure}

Figure \ref{fig-riskminimizers}B suggests a way to resolve this problem. It shows that the derivatives $d_{\text{sawtooth},a}'(x)$ are highly nonsmooth. Since it is the derivative that determines the penalty in $\hat{R}_1(d)$, this suggests that controlling the smoothness of $d'(x)$ can control overfitting. A natural choice is a total variation-based shape constraint on $d'(x)$. Indeed, Proposition \ref{prop-abscont} shows that as long as the total variation of $d'(x)$ is not too large, functions like $d_{\text{sawtooth},a}(x)$ cannot exist. 

\begin{proposition}
\label{prop-abscont}
Let $x_1 < \dots < x_n$ and $d:\mathbb{R}\to\mathbb{R}$ be an absolutely continuous function. If $TV(d') < 2(n-1)$, then there exists an $i \in \{1, \dots, n\}$ such that either $d(x_i) \neq x_i$ or $d'(x_i) \geq 0$.
\end{proposition}


Conveniently, the oracle minimizer $d^*(x)$ \eqref{oracle} has a derivative with manageable total variation. Proposition \ref{prop-optimalproperties} characterizes this quantity, along with other properties of $d^*(x)$ that we will use to design a function class over which we will minimize $\hat{R}_1(d)$. In particular, Proposition \ref{prop-optimalproperties} implies that $d^*(x)$ is always bound, absolutely continuous, and monotone non-decreasing. 

\begin{proposition}
\label{prop-optimalproperties}
The optimal separable estimator \eqref{oracle}, $d^*(x)$, is bounded, monotone non-decreasing, and Lipschitz continuous; its derivative has bound total variation and is also Lipschitz continuous. Let $r(\theta) = \max_{i=1}^n \theta_i - \min_{i=1}^n \theta_i$ denote the range of $\theta_1, \dots, \theta_n$. These bounds are summarized as:
\begin{enumerate}
\item $\min_{i=1}^n \theta_i \leq d^*(x) \leq \max_{i=1}^n \theta_i$
\item $0 \leq d^{*\prime}(x) \leq \sigma^{-2} r(\theta)^2$
\item $TV(d^{*\prime}) \leq \sigma^{-2} r(\theta)^2$
\item $\left\lvert d^{*\prime\prime}(x) \right\rvert \leq \sigma^{-4} r(\theta)^3$ .
\end{enumerate}
\end{proposition}

Using Propositions \ref{prop-abscont} and \ref{prop-optimalproperties}, we can now choose a suitable function class for estimating $d^*(x)$. Observe that if $TV(d^{*\prime}) < 2 (n-1)$, then we can find an upper bound on $TV(d')$ that simultaneously prevents over-fitting of $\hat{R}_1(d)$ and ensures that $d^*(x)$ is in our function class. Since neither $TV(d^{*\prime})$ nor $r(\theta) = \max_{i=1}^n \theta_i - \min_{i=1}^n \theta_i$ are known when defining our function class, we specify data-dependent estimates for these quantities so that both of our desired properties hold simultaneously with high probability. We find a high probability upper bound on the range of $\theta_1, \dots, \theta_n$ by expanding the range of $X_1, \dots, X_n$ by a non-random quantity, $b_n > 0$. We propose the following data-dependent function class:
\begin{equation}
\label{eq-D1}
\mathcal{D}_{1,n} = \left\{ d:\mathbb{R}\to[\min_i X_i - b_n, \max_i X_i + b_n] \mid d ~ \text{monotone non-decreasing}, TV(d') \leq \tau_n \right\},
\end{equation}
where $b_n$ and $\tau_n = \sigma^{-2}(\max_{i=1}^n X_i - \min_{i=1}^n X_i + 2 b_n)^2$ are constants that can depend on $X_i$, $\sigma$, and $n$. The inclusion of the monotonicity constraint is advantageous because from Proposition \ref{prop-optimalproperties}, $d^*(x)$ is always monotone non-decreasing. 
Note that the constraint $TV(d') \leq \tau_n < \infty$ ensures that all of the functions in $\mathcal{D}_{1,n}$ are absolutely continuous.

We now need to ensure that $\mathcal{D}_{1, n}$ contains the oracle estimator $d^*(x)$. Lemma \ref{lemma-prob} implies that this will be true with high probability if $b_n$ is of the order $\{\sigma^2 \log(n)\}^{1/2}$. 

\begin{lemma}
\label{lemma-prob}
Let $b_n \geq 0$ be a non-random value that may depend on $\sigma$ and $n$, $\mathcal{D}_{1,n}$ be defined as in \eqref{eq-D1}, and $d^*(x)$ be the optimal separable estimator \eqref{oracle}. Then
$$
P\left( d^* \not\in \mathcal{D}_{1,n} \right) \leq 6 n \exp\left(\frac{- b_n^2}{2 \sigma^2}\right) .
$$
\end{lemma}
In particular, when $b_n = (K \sigma^2 \log n)^{1/2}$ for some $K \in \mathbb{R}$, the bound becomes $6 n^{1-\frac{K}{2}}$, which can be made to converges to zero at an arbitrarily quick polynomial rate by selecting the correct $K > 2$. We therefore define the following estimate of $d^*(x)$ using the function class using our data-dependent function class $\mathcal{D}_{1,n}$:
\begin{equation}
\label{eq-estimator1}
\hat{d} = \arg\min_{d\in\mathcal{D}_{1,n}} \hat{R}_1(d) .
\end{equation}

Theorem \ref{thrm-pwl-implementation} shows that $\hat{d}(x)$ can be characterized as a piecewise linear function. The proof of Theorem \ref{thrm-pwl-implementation} uses Carathéodory's theorem for convex hulls \citep{Caratheodory1911} to reduce the infinite dimensional optimization problem over $\mathcal{D}_{1,n}$ to a finite dimensional one, similar to the proof of Theorem 1 in \citet{Boyd2018}. 
Notice that the range of $\hat{d}(x)$ always lies within $\left[\min_{i=1}^n X_i, \max_{i=1}^n X_i\right]$, so in practice the role of $b_n$ is only needed for defining the high probability bound $TV(d^{*\prime}) \leq \tau_n$. 

\begin{theorem}
\label{thrm-pwl-implementation}
For every $g\in\mathcal{D}_{1,n}$ \eqref{eq-D1}, there exists a continuous, piecewise linear function, $\tilde{g}\in\mathcal{D}_{1,n}$ with at most $n+3$ knots, such that $\tilde{g}(X_i) = g(X_i)$ for $i = 1, \dots, n$, $\sum_{i=1}^n \tilde{g}'(X_i) = \sum_{i=1}^n g'(X_i)$, and $\tilde{g}'(x) = 0$ for $x < \min_{i=1}^n X_i$ and $x > \max_{i=1}^n X_i$.
\end{theorem}

Based on Theorem \ref{thrm-pwl-implementation} we implement $\hat{d}(x)$ as a piecewise linear function: $\hat{d}(x) = \beta_0 + \sum_{j=1}^{M} \beta_j (x - x_j)_+$, where $(y)_+ = \max\{y, 0\}$ and $x_1 < x_2 < \dots < x_M$ are a finite collection of fixed knots. The minimization of $\hat{R}_1(d)$ over $\beta_i$ and $x_j$ is non-convex, so in practice we approximate $\hat{d}(x)$ by fixing a regular grid of knots and only optimizing over $\beta_i$. 
We note that for fixed knots $x_1 < x_2 < \dots < x_M$, minimizing $\hat{R}_1(d)$ over $\beta_i$ is a convex problem that can be readily solved by many off-the-shelf convex optimizers. Other potentially viable strategies for rapidly finding a good set of knots include placing knots at each of the observations or their quantiles and iterative approaches that slowly accumulate knots, similar to \citet{Boyd2018, BoydSchiebingerRecht2017}. The performance of various knot placement strategies are compare in Appendix A of the Supplementary Material.

Once the knots are fixed, we approximate $\hat{d}(x)$ as:
\begin{align*}
\hat{d}(x) &\approx
    \arg\min_{\beta_0, \beta_1, \dots, \beta_M \in \mathbb{R}} \sum_{i=1}^n \left\{X_i - \beta_0 - \sum_{j=1}^M \beta_j (X_i - x_j)_+ \right\}^2 + 2 \sigma^2 \sum_{i=1}^n \sum_{j=1}^M \beta_j 1(X_i \geq x_j) \\
& \text{subject to:} \quad \sum_{j=1}^M \beta_j = 0 , \quad \sum_{j=1}^k \beta_j \geq 0 ~ \text{for} ~ k = 1, \dots, M-1, \quad \sum_{j=1}^M \vert \beta_j \vert \leq \tau_n ,
\end{align*}
where the constraints correspond to our bounded, monotone non-decreasing, and total variation constraints, respectively. We note that there are two interesting interpretations of this parameterization. First, when we place the knots at each of the observations, this estimator is a constrained, irregularly-spaced, degree-one trend-filter \citep{Tibshirani2014}; the function constraints can be easily translated into the trend-filter parameterization. Second, this parameterization resembles a degree-one locally adaptive regression spines \citep{MammenvandeGeer1997} and one-dimensional MARS \citep{Friedman1991}; however, our minimization problem includes constraints and a slightly different penalty. We stick to the more general model described above for flexibility and ease of implementation.

Theorem \ref{thrm-pwl-rate} below shows that the conservative, high probability upper bounds on $TV(d^{*\prime})$ and $r(\theta)$ that we pursued in \eqref{eq-D1} produce an asymptotically optimal estimator. 
In practice, the choice of $\tau_n$ in the statement of Theorem \ref{thrm-pwl-rate} is usually too large and a smaller value may be more efficient; we describe some alternatives in Section \ref{section-sims}. 

\begin{theorem}
\label{thrm-pwl-rate}
Let $X_i \sim N(\theta_i, \sigma^2)$ for $i=1,\dots,n$ be independent such that $\sigma^2 > 0$ is known and assume $\max_{i=1}^n \vert \theta_i \vert < C_n$. Further, let $b_n = (3 \sigma^2 \log n)^{1/2}$ and $\tau_n = \sigma^{-2} (\max_{i=1}^n X_i - \min_{i=1}^n X_i + 2 b_n)^2$. Define $\hat{d}(x)$ as in \eqref{eq-estimator1} and let $d^*(x)$ be the optimal separable estimator \eqref{oracle}, then
$$
R(\hat{d}; \theta) - R(d^*; \theta) = \mathcal{O}\left\{ n^{-1/2} \left(C_n + \log^{1/2} n \right)^2 \right\} .
$$
\end{theorem}

The excess risk bounded in Theorem \ref{thrm-pwl-rate} is sometimes referred to as the compound regret \citep{Efron2019, PolyanskiyWu2021}. Our result implies that our constrained and penalized regression estimator can achieve asymptotically zero regret whenever $C_n = o(n^{1/4})$. 
When $C_n = o\left(n^{\epsilon}\right)$ for every $\epsilon > 0$, the regret of our estimator is essentially of order $n^{-1/2}$. This matches the best known rate for this simultaneous estimation problem over an $\ell_\infty$ ball, achieved by the nonparametric $g$-modeling estimator $\hat{d}_{JZ}$ of \citet{JiangZhang2009}, whose asymptotic ratio optimality result \eqref{ratio_opt} can be translated to a regret that also scales like $n^{-1/2}$ \citep{PolyanskiyWu2021}.

Our proof of Theorem \ref{thrm-pwl-rate} leverages standard techniques for studying M-estimators. This leads naturally to bounds on the regret. An alternative quantity of interest would be the ratio of the two risks, as in \eqref{ratio_opt}. Bounding the ratio is more meaningful in settings where the optimal separable estimator may perform very well, such as when $\theta$ is sparse and the oracle risk $R(d^*; \theta)$ already tends quickly toward zero. Studying the ratio optimality of our $\hat{d}$ would be an interesting direction for future work, and simulations in Section \ref{section-sims} already show that our estimator is comparable to the nonparametric empirical Bayes methods when estimating sparse $\theta$.

In our proof we use symmetrization and metric entropy bounds based on the size of $\mathcal{D}_{1,n}$ and the function class of its derivatives. Our total variation bound, based on Proposition \ref{prop-abscont}, provides the needed control over the function class of derivatives, while the bounded and monotone non-decreasing constraints are used to control the size of $\mathcal{D}_{1,n}$. We note that the monotonicity constraint is not necessary for the proof and other function classes such as bound total variation or Lipschitz over a finite domain, may be used instead. However, these alternative function classes may increase the size of the function class, introduce additional tuning parameters, or slow the rate of excess risk convergence.

\section{Removing Constraints with Biased Risk Minimization}
\label{section-gkde}

Our estimator $\hat{d}(x)$ \eqref{eq-estimator1} above has both a monotonicity as well as a total variation constraint, and the latter is crucial to avoid overfitting. In contrast, in standard isotonic least-squares regression, a monotonicity constraint alone already offers sufficient regularization \citep{vandeGeer2000}. A natural question is whether monotonicity alone might be sufficient for developing an asymptotically optimal estimator in the simultaneous estimation problem.

Investigating this question reveals an interesting connection between the constraints on $d(x)$ and the estimate of the oracle loss $R(d; \theta)$ \eqref{oracle-loss}.
A closer inspection of the decomposition of this loss in \eqref{risk-sum} gives
\begin{align}
R(d; \theta) 
    &= \frac{1}{n} \sum_{i=1}^n E \{X_i - d(X_i)\}^2 + 2 \sigma^2 \int_{-\infty}^{\infty} d'(x) f_\theta(x) dx - \sigma^2 , \label{risk-integral}
\end{align}
where $f_\theta(x) = n^{-1} \sum_{i=1}^n \phi_\sigma(x - \theta_i)$. Previously, we replaced the integral above with an empirical expectation to give $\hat{R}_1(d)$ \eqref{SURE}, and we found that monotone functions could still overfit. The core problem was that the empirical expectation penalty could not control the shape of $d(x)$ for $x \ne X_i$ because the values of $d'(x)$ for these $x$ do not affect the risk estimate; recall Figure \ref{fig-riskminimizers}. In contrast, the integral in \eqref{risk-integral} controls the shape of $d(x)$ for all $x$ in the support of $f_\theta(x)$, because all values of $d'(x)$ would affect the risk.

This suggests that a monotonicity constraint would be sufficient to achieve asymptotical optimality given only a monotonicity constraint if we could find a continuous estimate $f_\theta(x)$ 
so that $d'(x)$ is penalized over the whole domain of the estimator. 
Drawing on the kernel density estimator literature \citep{Tsybakov2009}, we use
\begin{equation}
\label{eq-kde}
\hat{f}_h(x) = \frac{1}{n} \sum_{i=1}^n \phi_h(x - X_i)
\end{equation}
for some bandwidth $h > 0$. 
In the Bayesian setting, $\hat{f}_h(x)$ is equivalent to a Gaussian kernel estimate of the density $X_i$ marginalizing over the random $\theta_i$; however, we stress that our approach does not depend on this interpretation. 

We propose the following function class for fitting estimators:
\begin{equation}
\label{eq-D0}
\mathcal{D}_{0,n} = \left\{ d:\mathbb{R}\to[\min_i X_i - b_n, \max_i X_i + b_n] \mid d ~ \text{monotone non-decreasing} \right\},
\end{equation}
where $b_n$ is a constant depending on $\sigma$ and $n$. As in Lemma \ref{lemma-prob}, $b_n$ is used to ensure that $\mathcal{D}_{0,n}$ contains the optimal separable estimator with high probability. This result can be established by noting that $\mathcal{D}_{1,n} \subset \mathcal{D}_{0,n}$ when the range of $\mathcal{D}_{0,n}$ is at least as large as range of $\mathcal{D}_{1,n}$. We would like to use plug our $\hat{f}_h(x)$ \eqref{eq-kde} into the loss decomposition \eqref{risk-integral} and minimize the result over functions in $\mathcal{D}_{0,n}$. Unfortunately, this is not well-justified because not all monotone functions are weakly differentiable, so Stein's lemma cannot be directly applied. In order to justify a risk estimate for all functions in $\mathcal{D}_{0,n}$, Proposition \ref{prop-extendedSL} extends Stein's lemma to apply to every bounded, monotone non-decreasing function. The proof of Proposition \ref{prop-extendedSL} draws heavily on a Corollary 1 in \cite{Tibshirani-SL}. Rather than extend Stein's lemma to apply to all of $\mathcal{D}_{0,n}$, we could instead study a smoother function class that is a subset of $\mathcal{D}_{0,n} \cap \{d \mid d ~ \text{weakly differentiable}\}$; however, this may introduce additional tuning parameters or make it difficult to establish a parameterization for the optimal estimator.

\begin{proposition}
\label{prop-extendedSL}
Let $Z \sim N(\mu,\sigma^2)$. Let $g:\mathbb{R}\to\mathbb{R}$ have finite total variation and let $\mathcal{J}(g) = \{t_1, t_2, \dots \}$ be the countable set of locations where $g$ has a discontinuity. Let $g'$ be the derivative of $g$ almost everywhere and assume that $E\vert g'(Z) \vert < \infty$. Then
$$
\frac{1}{\sigma^2} E[(Z-\mu) g(Z-\mu)] =  E[g'(Z-\mu)] + \sum_{t_k \in \mathcal{J}(g)} \phi_\sigma(t_k-\mu) \left\{ \lim_{x \downarrow t_k} g(x) - \lim_{x \uparrow t_k} g(x) \right\} .
$$
\end{proposition}

Since every bounded, monotone non-decreasing function has finite total variation, we can combine Proposition \ref{prop-extendedSL} with \eqref{risk-integral} and \eqref{eq-kde} to find an estimate of the oracle loss $R(d; \theta)$ \eqref{oracle-loss} that holds for every $d \in \mathcal{D}_{0,n}$:
\begin{equation}
  \label{risk-biasedEstimate} 
  \begin{aligned}
    \hat{R}_0(d; h) = \frac{1}{n} \sum_{i=1}^n \{X_i &- d(X_i)\}^2 
    + 2 \sigma^2 \int_{-\infty}^{\infty} d'(x) \hat{f}_{h}(x) dt \\
    &+ 2 \sigma^2 \sum_{x_k \in \mathcal{J}(d)} \hat{f}_{h}(x_k) \left\{ \lim_{x \downarrow x_k} d(x) - \lim_{x \uparrow x_k} d(x) \right\} - \sigma^2 .
  \end{aligned}
\end{equation}
This risk estimate is biased for fixed $h$, but minimizing it gives a useful estimate of $d^\star(x)$ when $h \downarrow 0$ at an appropriate rate with $n$.
Because $\hat{f}_{h}(x)$ is positive for all of $\mathbb{R}$ and because only constant functions minimize both the penalty terms, $\hat{R}_0(d; h)$ almost surely cannot be over-fit by functions in $\mathcal{D}_{0,n}$.

We propose the following estimator using this new risk estimate:
\begin{equation}
\label{eq-estimator0}
\tilde{d}_h = \arg\min_{d\in\mathcal{D}_{0,n}} \hat{R}_0(d;h) .
\end{equation}
Theorem \ref{thrm-pwc-implementation} below exactly characterizes $\tilde{d}_h(x)$ for any fixed value of $h > 0$. This theorem is similar to Proposition 1 in \citet{MammenvandeGeer1997} except they study the standard total variation penalty and place knots at the data, while we study a weighted total variation penalty and place knots according to the weight function. In practice, the penalized least-squares regression formulation ensures that we can ignore $b_n$ because $\tilde{d}_h(x)$ always has a range within $[\min_{i=1}^n X_i, \max_{i=1}^n X_i]$, so $\mathcal{D}_{0,n}$ is free of tuning parameters.

\begin{theorem}
\label{thrm-pwc-implementation}
Let $h > 0$. For every $g\in\mathcal{D}_{0,n}$ \eqref{eq-D0} that is right-continuous at $X_1, \dots, X_n$, there exists a piecewise constant function, $\tilde{g}\in\mathcal{D}_{0,n}$, that has at most $n-1$ knots and satisfies $\tilde{g}(X_i) = g(X_i)$ from the right and $\hat{R}_0(\tilde{g}; h) \leq \hat{R}_0(g; h)$. The knots of $\tilde{g}(x)$ lie at the minimum of $\hat{f}_h(x)$ between each consecutive order statistics of $X_1, \dots, X_n$.
\end{theorem}

Theorem \ref{thrm-pwc-implementation} provides both the existence and the finite parameterization of $\tilde{d}_h(x)$. This makes exactly finding $\tilde{d}_h(x)$ a convex optimization problem for any fixed $h > 0$, unlike $\hat{d}(x)$ in the previous section. For a fixed $h > 0$, finding the true knots can still be a somewhat time-intensive procedure and in practice placing the knots at the average of each consecutive pair of $X_i$ provides a good approximation and alleviates much of the computational burden. We further note that the true knots may not be worth finding because they minimize $\hat{R}_0(d; h)$, so they are random and may not be the best knots for minimizing the out-of-sample risk. Simulations comparing knot placement strategies are in Appendix A of the Supplementary Material. Once the knots are fixed, computing $\tilde{d}_h(x)$ can be performed using off-the-shelf convex optimization software. 

We efficiently compute $\tilde{d}_h(x)$ using a tailored algorithm that builds on Theorem \ref{thrm-pwc-implementation} and leverages the pool adjacent violators algorithm \citep{pava, BestChakravarti1990}. The goal is to transform \eqref{eq-estimator0} into an unpenalized isotonic regression problem. We assume that $X_1 < \dots < X_n$ without loss of generality. We first note that $\tilde{d}_h(x)$ is a piecewise constant function, by Theorem \ref{thrm-pwc-implementation}, so the integral term in $\hat{R}_0(\tilde{d}_h; h)$ \eqref{risk-biasedEstimate} is zero and our penalty is linear. 
We adopt a degree-zero trend-filter parameterization \citep{Tibshirani2014} for $\tilde{d}_h(x)$, which 
letting $\tilde{d}_h(x) = \beta_0 + \sum_{i=1}^{n-1} \beta_i 1(x \geq x_i)$ and setting $\beta_0 = d(X_{1})$ and $\beta_i = d(X_{i+1}) - d(X_{i})$. Now, for a fixed $h > 0$ and knots $x_1 < \dots< x_{n-1}$ from Theorem \ref{thrm-pwc-implementation}, we define
\begin{align*}
\mathbf{x} &= \begin{bmatrix} X_1 \\ \vdots \\ X_n \end{bmatrix}, &
\mathbf{d} &= \begin{bmatrix} d(X_1) \\ \vdots \\ d(X_n) \end{bmatrix}, &
\mathbf{f}_h &= \begin{bmatrix} \hat{f}_h(x_1) \\ \vdots \\ \hat{f}_h(x_{n-1}) \end{bmatrix}, &
\mathbf{D} &= \begin{bmatrix} 1 & -1 & 0 & \dots & 0 & 0 \\ 0 & 1 & -1 & \dots & 0 & 0 \\ \vdots & \vdots & \vdots & \ddots & \vdots & \vdots \\  0 & 0 & 0 &\dots & 1 & -1 \end{bmatrix}, 
\end{align*}
and express \eqref{eq-estimator0} as
\begin{equation}
\label{eq-penaltyabsorption}
\mathbf{\tilde{d}_h} = \arg\min_{d \in \mathcal{D}_{0,n}} \Vert \mathbf{x} - \mathbf{d} \Vert_2^2 + 2 n \sigma^2 \mathbf{w}_h' \mathbf{d} - n \sigma^2 = \arg\min_{d \in \mathcal{D}_{0,n}} \left\Vert \left( \mathbf{x} - n \sigma^2 \mathbf{w}_h \right) - \mathbf{d} \right\Vert_2^2 ,
\end{equation}
where $\mathbf{w}_h = \mathbf{D}' \mathbf{f}_h$ and the last equality absorbs the linear penalty into the least-squares loss, motivated by the work of \citet{WuMeyerOpsomer2015} on penalized isotonic regression. Since our minimization is now a monotone least-squares regression problem, the pooled adjacent violators algorithm can be used to efficiently compute $\tilde{d}_h(X_i)$ for each $i = 1, \dots, n$. 
Interestingly, \eqref{eq-penaltyabsorption} can be interpreted as an unpenalized isotonic regression of sorted and suitably modified $X_i$.

Theorem \ref{thrm-pwc-rate} shows that when $h > 0$ is properly chosen, our $\tilde{d}_h(x)$ has asymptotically optimal regret. This result identifies a rate for $h \to 0$ that optimally trades-off the bias and variance of the risk estimate. In simulations, this asymptotic rate seems works well as a default value for $h$, and other standard kernel density bandwidth selection methods, such as Silverman's rule-of-thumb and cross-validation of the mean integral squared error \citep{Silverman1986, Tsybakov2009} can be used as well. 

\begin{theorem}
\label{thrm-pwc-rate}
Let $X_i \sim N(\theta_i, \sigma^2)$ be independent such that $\sigma^2$ is known and $\max_i \vert \theta_i \vert < C_n$. Further, let $b_n = \{(8/3) \sigma^2 \log(n) \}^{1/2}$ and $h_n > 0$ such that $h_n \asymp \sigma n^{-1/6}$. Define $\tilde{d}_{h_n}(x)$ as in \eqref{eq-estimator0} and let $d^*(x)$ be the optimal separable estimator \eqref{oracle}, then
$$
R(\tilde{d}_{h_n}; \theta) - R(d^*; \theta) = \mathcal{O}\left\{ n^{-1/3} \left(C_n + \log^{1/2} n \right)^2 \right\} .
$$
\end{theorem}

This result demonstrates that monotonicity alone can produce an asymptotically optimal estimator; however, the use of a larger function class results in a slower convergence of excess risk, likely due to the bias in the risk estimate \eqref{risk-biasedEstimate}. 
On the other hand, $\tilde{d}_h(x)$ can be approximately computed in $\mathcal{O}(n \log n)$ time by placing knots at the average of consecutive order statistics and using the fast Fourier transformation to approximate $\hat{f}_h(x)$. This suggests that $\tilde{d}_h(x)$ can be useful for large simultaneous estimation problems.

It is also possible to achieve asymptotic regret after removing the monotonicity shape constraint. Corollary \ref{corollary-pwc-rate} below demonstrates that the class of bounded functions with bounded total variation is sufficient. The tradeoff is that this introduces an additional tuning parameter, though this can be estimated using a conservative high probability upper bound as we did in Section \ref{section-ecdf}. 
Removing the monotonicity constraint means that more general convex optimization software is required to fit the estimator; however, Theorem \ref{thrm-pwc-implementation} still applies to this larger function class. The proof of Corollary \ref{corollary-pwc-rate} follows almost immediately from the proof of Theorem \ref{thrm-pwc-rate}. 

\begin{corollary}
\label{corollary-pwc-rate}
Let $X_i \sim N(\theta_i, \sigma^2)$ be independent such that $\sigma^2$ is known and $\max_{i=1}^n \vert \theta_i \vert < C_n$. Further, let $b_n = \{(8/3) \sigma^2 \log(n) \}^{1/2}$, $\lambda_n = \max_{i=1}^n X_i - \min_{i=1}^n X_i + 2 b_n$, and $h_n > 0$ such that $h_n \asymp \sigma n^{-1/6}$. Define the new function class: 
$$
\mathcal{C}_{0,n} = \left\{ d:\mathbb{R}\to[\min_i X_i - b_n, \max_i X_i + b_n] \mid TV(d) \leq \lambda_n \right\} , 
$$
and the corresponding estimator $\bar{d}_{h_n} = \arg\min_{d\in\mathcal{C}_{0,n}} \hat{R}_0(d; h_n)$. Let $d^*(x)$ be the optimal separable estimator \eqref{oracle}, then
$$
R(\bar{d}_{h_n}; \theta) - R(d^*; \theta) = \mathcal{O}\left\{ n^{-1/3} \left(C_n + \log^{1/2} n \right)^2 \right\} .
$$
\end{corollary}

The estimators $\bar{d}(x)$ and $\tilde{d}_h(x)$ exhibit two interesting connections with empirical Bayes $f$-estimators. Firstly, as noted by \citet{KoenkerMizera2014}, the monotonicity constraint is not necessary but it can improve the efficiency of nonparametric empirical Bayes estimators and removes a tuning parameter. Secondly, the monotone $f$-estimator proposed by \citet{KoenkerMizera2014} is a piecewise constant estimator, just like $\tilde{d}_h(x)$; however, because our $\hat{f}_h(x)$ is not necessarily a log-concave density, there is no guarantee that the two estimators are equal for any value of $h > 0$. 

\section{Simulations}
\label{section-sims}

\subsection{Methods}

We compared the performances of our constrained SURE estimator \eqref{eq-estimator1}, $\hat{d}(x)$, and monotone-only estimator \eqref{eq-estimator0}, $\tilde{d}_h(x)$, to those of state-of-the-art nonparametric empirical Bayes estimators for the canonical Gaussian sequence problem. Specifically, we compared to the $g$-modeling $\hat{d}_{JZ}(x)$ of \citet{JiangZhang2009},  using 300 regularly spaced knots along $[\min_{i=1}^n X_i, \max_{i=1}^n X_i]$, and the $f$-modeling $\hat{d}_{BG}(x)$ of \citet{BrownGreenshtein2009}, using their recommended bandwidth $h = \sigma (\log n)^{-1/2}$.

When fitting our $\tilde{d}_h(x)$, we made no attempt to tune the bandwidth and instead set it to its theoretically optimal rate from Theorem \ref{thrm-pwc-rate}. Alternative bandwidth choices are explored in Appendix A of the Supplementary Material. When fitting our $\hat{d}(x)$, we pursued three distinct methods for selecting the total-variation constraint $TV(d') \leq \tau_n$. The first method was to set $\tau_n$ equal to the high-probability upper bound $\sigma^{-2}\{ \max_{i=1}^n X_i - \min_{i=1}^n X_i + (3 \sigma^2 \log n)^{1/2} \}^2$ studied in Theorem \ref{thrm-pwl-rate}. The second method was to develop a plug-in estimate of $TV(d^{*\prime})$. We first estimated
$$
d^{*\prime\prime}(x) = \sigma^2 \frac{f_\theta^{(3)}(x)}{f_\theta(x)} - 3 \sigma^2 \frac{f^{\prime\prime}(x) f^{\prime}(x)}{f_\theta(x)^2} + 2 \sigma^2 \left\{\frac{f^{\prime}(x)}{f_{\theta}(x)}\right\}^3 
$$
by plugging-in $X_i$ for $\theta_i$ to estimate each of the $f^{(k)}_\theta(t)$ for each $k \in\{0,1,2,3\}$, then we used numerical integration to approximate $TV(d^{*\prime}) = \int_{-\infty}^{\infty} \vert d^{*\prime\prime}(x) \vert dx$. In practice this often out-performed the conservative high-probability bound. The final method was to choose $\tau_n$ using $k$-fold cross-validation, where we estimated out-of-sample risk using  $\hat{R}_1(d)$ \eqref{SURE} applied to the $\hat{d}(x)$ fitted on the training folds. In the results below we used $k = 5$.

We fit $\tilde{d}_h(x)$ using knots at the average of each consecutive order statistic of $X_1, \dots, X_n$, and we fit $\hat{d}(x)$ using 30 regularly spaced knots along $(\min_{i=1}^n X_i, \max_{i=1}^n X_i)$. Alternative knot placement strategies are explored in Appendix A. We provide \texttt{R} implementations of our methods in the \texttt{cole} package (https://github.com/sdzhao/cole). Figure \ref{fig-methods} visualizes each of these estimators applied to the same data, along with the optimal estimator $d^\star(x)$ \eqref{oracle} as a reference.

\begin{figure}[t!]
  \includegraphics[width=\linewidth]{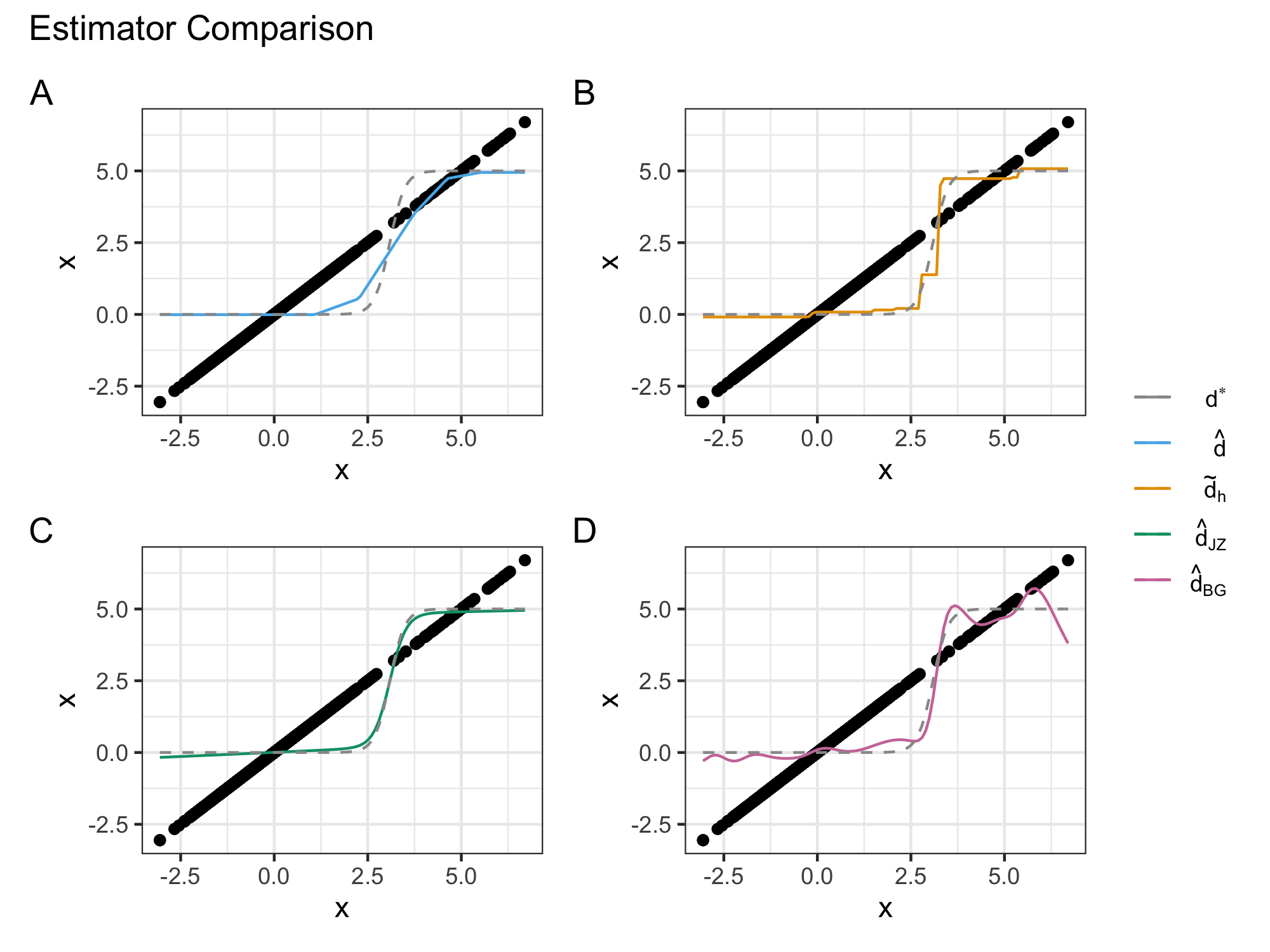}
  \caption{Visual comparison of different estimators on a sample of $n = 1000$ data points. Solid points correspond to observed $X_i$. $d^\star$: oracle estimator \eqref{oracle}; $\hat{d}$: proposed constrained SURE minimizer \eqref{eq-estimator1} using a plug-in estimate for $\tau_n$; $\tilde{d}_h$: proposed monotone-only estimator \eqref{eq-estimator0} using its theoretically optimal bandwidth; $\hat{d}_{JZ}$: estimator of \citet{JiangZhang2009}; $\hat{d}_{BG}$: estimator of \citet{BrownGreenshtein2009}.}
  \label{fig-methods}
\end{figure}

\subsection{High-Noise Setting}

We first compared these methods under a standard high-noise setting where $X_i \sim N(\theta_i, \sigma^2)$ and $\sigma^2 = 1$. We fixed $n = 1000$ and considered sparse configurations, where many of the $\theta_i$ were zero, and dense configurations, where all of the $\theta_i$ were nonzero. Estimation of sparse mean vectors has been thoroughly studied in the literature, for example by \citet{JohnstoneSilverman2004}, but there are many problems of interest, such as denoising gene expression levels from single-cell RNA-sequencing data \citep{huang2018saver, li2018accurate}, 
that are not well captured by the sparse model. Throughout, we evaluated the performance of each estimator using the sum of squared errors $\sum_{i=1}^n (\theta_i - \hat{\theta}_i)^2$, where $\hat{\theta}_i$ is the corresponding estimate of $\theta_i$, averaged over 50 replications.

\begin{table}[t!]
\centering
\begin{tabular}{l|r|r|r|r|r|r|r|r|r|r|r|r}
\hline
\multicolumn{1}{c|}{nonzero} & \multicolumn{4}{c|}{5} & \multicolumn{4}{c|}{50} & \multicolumn{4}{c}{500} \\
\cline{1-1} \cline{2-5} \cline{6-9} \cline{10-13}
\multicolumn{1}{c|}{$\mu$} & \multicolumn{1}{c|}{3} & \multicolumn{1}{c|}{4} & \multicolumn{1}{c|}{5} & \multicolumn{1}{c|}{7} & \multicolumn{1}{c|}{3} & \multicolumn{1}{c|}{4} & \multicolumn{1}{c|}{5} & \multicolumn{1}{c|}{7} & \multicolumn{1}{c|}{3} & \multicolumn{1}{c|}{4} & \multicolumn{1}{c|}{5} & \multicolumn{1}{c}{7} \\
\cline{1-1} \cline{2-2} \cline{3-3} \cline{4-4} \cline{5-5} \cline{6-6} \cline{7-7} \cline{8-8} \cline{9-9} \cline{10-10} \cline{11-11} \cline{12-12} \cline{13-13}
$\hat{d}_{UB}$ & 84 & 66 & 58 & 38 & 214 & 154 & 88 & 50 & 528 & 357 & 176 & 57\\
\hline
$\hat{d}_{PI}$ & 41 & 39 & 33 & 24 & 169 & 128 & 69 & 29 & 500 & 345 & 166 & 38\\
\hline
$\hat{d}_{CV}$ & 47 & 43 & 36 & 21 & 171 & 123 & 65 & 31 & 479 & 311 & 145 & 41\\
\hline
$\tilde{d}_h$ & 42 & 37 & 31 & 17 & 179 & 126 & 65 & 25 & 485 & 316 & 150 & 33\\
\hline
$\hat{d}_{JZ}$ & 39 & 34 & 23 & 11 & 157 & 105 & 58 & 14 & 459 & 285 & 139 & 18\\
\hline
$\hat{d}_{BG}$ & 53 & 49 & 42 & 27 & 179 & 136 & 81 & 40 & 484 & 302 & 158 & 48\\
\hline
\end{tabular}
\caption{\label{table-sparse}
  Sums of squared errors for sparse normal mean vectors under high noise.  $\hat{d}_{UB}$, $\hat{d}_{PI}$, and $\hat{d}_{CV}$: proposed constrained SURE estimator \eqref{eq-estimator1} with high-probability upper bound on $TV(d^{*\prime})$, plug-in estimate of $TV(d^{*\prime})$, and cross-validated $\tau_n$, respectively; $\tilde{d}_h$: proposed monotone-only estimator \eqref{eq-estimator0} using its theoretically optimal bandwidth; $\hat{d}_{JZ}$: estimator of \citet{JiangZhang2009}; $\hat{d}_{BG}$: estimator of \citet{BrownGreenshtein2009}.}
\end{table}

In the sparse configuration, we followed \citet{JohnstoneSilverman2004} and set $k \in \{5,50,500\}$ of the $\theta_i = \mu$ for $\mu \in \{3,4,5,7\}$, letting the remaining $\theta_i$ to zero. Table \ref{table-sparse} reports the results, where the errors for the estimators of \citet{JiangZhang2009} and \citet{BrownGreenshtein2009} were copied from the corresponding papers. For $\hat{d}_{JZ}(x)$ we estimated the prior using convex optimization, as described by \citet{KoenkerMizera2014} and implemented in the REBayes R package of \citet{koenker2017rebayes}. Our nonparametric regression estimators performed comparably to the nonparametric empirical Bayes estimators over a range of sparse settings. In most settings, $\hat{d}_{JZ}(x)$ had the best performance, followed closely by our $\tilde{d}_h(x)$, then our $\hat{d}(x)$, and finally by $\hat{d}_{BG}(x)$. 

\begin{table}[t!]
\centering
\begin{tabular}{l|r|r|r|r}
\hline
 & Uniform(0, 5) & N(0,1) & Laplace(0,1) & 0.5 N(-3,1) + 0.5 N(3,2) \\
\hline
$\hat{d}_{UB}$ & 709 & 566 & 546 & 853\\
\hline
$\hat{d}_{PI}$ & 647 & 514 & 501 & 794\\
\hline
$\hat{d}_{CV}$ & 653 & 520 & 507 & 800\\
\hline
$\tilde{d}_h$ & 670 & 532 & 523 & 834\\
\hline
$\hat{d}_{JZ}$ & 646 & 511 & 497 & 795\\
\hline
$\hat{d}_{BG}$ & 674 & 535 & 526 & 838\\
\hline
\end{tabular}
\caption{\label{table-dense}
  Sums of squared errors for dense normal mean vectors under high noise. Estimators are the same as those in Table \ref{table-sparse}.}
\end{table}

In the dense configuration, we generated $\theta_i$ by sampling them from an absolutely continuous prior distribution $G$ and then fixed them across replications. We considered a diverse collection of $G$: 1) the uniform distribution on $[0,5]$; 2) the Normal distribution with mean zero and variance of one; 3) the Laplace distribution with mean zero and variance of one; and 4) an asymmetric two-component Gaussian mixture with means at -3 and 3 and and variances of 1 and 2. These represent various tail bounds and problem complexities. In this scheme, the tail bound $E \max_{i=1}^n \vert \theta_i \vert$ plays the role similar to $C_n$ in our excess risk rate theorems. Table \ref{table-dense} reports the results and demonstrates a similar ranking of estimator performance as Table \ref{table-sparse}, though our $\hat{d}(x)$ with the plug-in estimate of $\tau_n$ works better than our $\tilde{d}_h(x)$ here. Our proposed estimators were more competitive with $\hat{d}_{JZ}(x)$ in dense configurations compared to the sparse configurations.

\subsection{\label{section-low-noise}Low-Noise Setting}

We also compared estimators under a low noise setting where we repeated the simulations in Tables \ref{table-sparse} and \ref{table-dense} but set $\sigma^2 = 10^{-8}$. As mentioned in Section \ref{section-intro}, this setting mimics more complex simultaneous estimation problems where observations have densities that can take extremely small values on most of the support of the unknown parameters. The introduction gave an example arising from multivariate linear regression. As another illustration, \citet{saha2020on} and \citet{soloff2021multivariate} studied nonparametric empirical Bayes estimation of $d$-dimensional $\theta_i$. In particular, if the observations $X_i \sim N(\theta_i, I)$, where $I$ is the $d \times d$ identity matrix, they assume $\theta_i \sim G(t)$ and estimate $G(t)$ by maximizing
\[
  \prod_{i = 1}^n \int \frac{1}{(2 \pi)^{1/2}} \exp\left\{-\frac{1}{2} \sum_{j = 1}^d (X_{ij} - t_j)^2\right\} dG(t)
\]
over discrete distributions supported on a pre-specified grid, where $X_{ij}$ is the $j$th coordinate of $X_i$. However, when $d$ is large, the density in the integrand can be extremely small at most of the grid points, which can cause difficulties for convex optimization software.

We encountered this exact difficulty in our low-noise univariate Gaussian sequence problem. The nonparametric empirical Bayes methods became less reliable. For example, in our simulations we tried to implement $\hat{d}_{JZ}(x)$ using convex optimization, but the underlying MOSEK optimizer frequently failed. We resorted to the EM approach of \citet{JiangZhang2009}, and even there we needed to artificially left-censor low density values to equal the smallest positive floating-point number supported by our machine. In our results we report only the results of our EM-based implementation. Our plug-in estimate of $TV(d^{*\prime})$ suffered from numerical stability issues in this regime as well, so we do not report its performance.

\begin{table}[t!]
\centering
\begin{tabular}{l|r|r|r|r|r|r|r|r|r|r|r|r}
\hline
\multicolumn{1}{c|}{nonzero} & \multicolumn{4}{c|}{5} & \multicolumn{4}{c|}{50} & \multicolumn{4}{c}{500} \\
\cline{1-1} \cline{2-5} \cline{6-9} \cline{10-13}
\multicolumn{1}{c|}{$\mu$} & \multicolumn{1}{c|}{3} & \multicolumn{1}{c|}{4} & \multicolumn{1}{c|}{5} & \multicolumn{1}{c|}{7} & \multicolumn{1}{c|}{3} & \multicolumn{1}{c|}{4} & \multicolumn{1}{c|}{5} & \multicolumn{1}{c|}{7} & \multicolumn{1}{c|}{3} & \multicolumn{1}{c|}{4} & \multicolumn{1}{c|}{5} & \multicolumn{1}{c}{7} \\
\cline{1-1} \cline{2-2} \cline{3-3} \cline{4-4} \cline{5-5} \cline{6-6} \cline{7-7} \cline{8-8} \cline{9-9} \cline{10-10} \cline{11-11} \cline{12-12} \cline{13-13}
$\hat{d}_{UB}$ & 3.42 & 4.74 & 4.79 & 5.45 & 2.75 & 5.07 & 5.15 & 5.22 & 2.09 & 5.07 & 5.99 & 5.57\\
\hline
$\hat{d}_{CV}$ & -17.76 & -17.43 & -17.58 & -17.51 & -17.41 & -17.82 & -17.79 & -17.68 & -17.69 & -17.68 & -17.85 & -17.88\\
\hline
$\tilde{d}_h$ & -17.78 & -17.53 & -17.70 & -17.70 & -17.95 & -17.82 & -17.78 & -17.69 & -17.70 & -17.72 & -17.84 & -17.88\\
\hline
$\hat{d}_{BG}$ & -11.52 & -11.52 & -11.52 & -11.51 & -11.50 & -11.51 & -11.50 & -11.50 & -11.51 & -11.52 & -11.51 & -11.52\\
\hline
$\hat{d}_{JZ}$ & -9.10 & -9.21 & -9.14 & -9.12 & -9.19 & -9.17 & -9.15 & -9.15 & -9.25 & -9.29 & -9.27 & -9.28\\
\hline
\end{tabular}\caption{\label{table-sparse-low-var}Natural logs of sums of squared errors for sparse normal mean vectors under low noise. The estimators used in this table are the same as those used in Table \ref{table-sparse}.}
\end{table}

\begin{table}[t!]
\centering
\begin{tabular}{l|r|r|r|r}
\hline
method & Uniform(0, 5) & N(0,1) & Laplace(0,1) & 0.5 N(-3,1) + 0.5 N(3,2) \\
\hline
$\hat{d}_{UB}$ & 3.24 & 7.83 & 7.50 & 8.76\\
\hline
$\hat{d}_{CV}$ & -1.40 & -4.11 & -2.42 & -1.46\\
\hline
$\tilde{d}_h$ & -10.93 & -11.12 & -11.16 & -11.05\\
\hline
$\hat{d}_{BG}$ & -11.52 & -11.50 & -11.51 & -11.53\\
\hline
$\hat{d}_{JZ}$ & 6.97 & 6.38 & 6.44 & 9.19\\
\hline
\end{tabular}
\caption{\label{table-dense-low-var}Natural logs of sums of squared errors for dense normal mean vectors under low noise. The estimators used in this table are the same as those used in Table \ref{table-dense}.}
\end{table}

Tables \ref{table-sparse-low-var} and \ref{table-dense-low-var} report the performances of each estimator for sparse and dense configurations, respectively. Here we quantify performance using $\log \sum_{i=1}^n (\theta_i - \hat{\theta}_i)^2$ averaged over 50 replications, where we take natural log because the sum of squared errors is otherwise already very low with such small noise. The results demonstrate that some of our proposed methods, with appropriate tuning, can match or outperform the state-of-the-art nonparametric empirical Bayes methods when $\sigma$ is small. Specifically, our monotone-only $\tilde{d}_h(x)$ \eqref{eq-estimator0} was the top performer in the sparse configurations and nearly matched the estimator of \citet{BrownGreenshtein2009} in the dense configurations. The $g$-modeling approach of \citet{JiangZhang2009}, which was the top performer in the high-noise settings, had difficulty with such small $\sigma^2$.

\section{Data analysis}
\label{section-analysis}

We applied our proposed approaches to denoise spatial transcriptomic data. Spatial transcriptomics is a recently developed set of experimental techniques that can profile gene expression from regions of intact tissue sections at very high spatial resolution \citep{moffitt2022emerging}. This additional spatial information has the potential to revolutionize our understanding of biological processes and is one reason why these technologies were named ``Method of the Year'' in 2021 by the journal \textit{Nature Methods} \citep{marx2021method}. Here we study data from a segment of the mouse small intestine, collected by \citet{petukhov2022cell} using a spatial transcriptomic technique called MERFISH \citep{chen2015spatially}.

Despite its additional spatial context, these technologies still suffer from issues of detection efficiency and other experimental errors. The resulting measurements can thus be viewed as noisy observations of true expression levels. It has been demonstrated that denoising raw gene expression values can be a useful preprocessing step in non-spatial single-cell RNA-sequencing \citep{eraslan2019single, huang2018saver, li2018accurate} and very recently in spatial transcriptomics \citep{wang2022region}.

In this section we apply our proposed simultaneous estimation methods to denoise MERFISH gene expression data from the mouse ileum. To formulate the problem, for a given cell let $Y_i$ denote the observed expression of gene $i$, for each of the $i = 1, \ldots, 241$ genes measured by \citet{petukhov2022cell}. We modeled the $Y_i$ as Poisson-distributed random variables and applied the Anscombe transform $X_i = 2 (Y_i + 3 / 8)^{1/2}$ to obtain roughly normally-distributed $X_i$ with unit variance. Our goal was to denoise the $X_i$ by simultaneously estimating the $\theta_i = E(X_i)$. Ideally we would measure the performance of estimates $\hat{\theta}_i$ using the mean squared error $\sum_i (\theta_i - \hat{\theta}_i)^2 / 241$, but the true $\theta_i$ are unknown. We therefore approximated $\theta_i$ using the average expression of gene $i$ over the three cells closest to the cell being denoised. We denoised 1,000 randomly selected cells, shown in Figure \ref{fig-ileum}.

\begin{figure}[t!]
  \centering
  \includegraphics[width=0.6\linewidth]{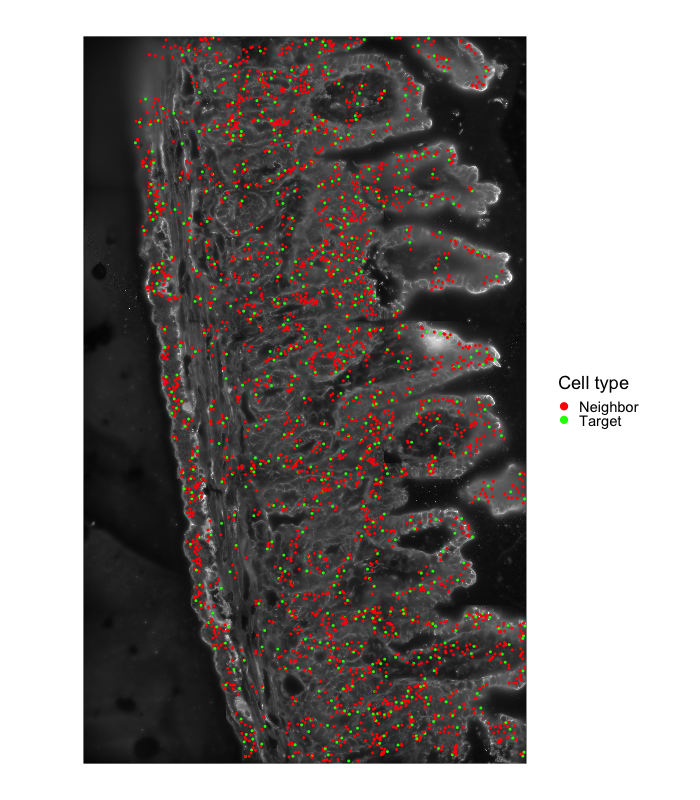}
  \caption{\label{fig-ileum}Spatial transcriptomic data from the mouse ileum \citep{petukhov2022cell}. One thousand target cells (green) were randomly selected to be denoised, and the three nearest neighbors of each cell (red) were used to assess denoising accuracy.}
\end{figure}

We applied the following estimators: the observed $X_i$, our proposed regression-based estimators $\hat{d}_{CV}(x)$ \eqref{eq-estimator1} and $\tilde{d}_h(x)$ \eqref{eq-estimator0}, and the state-of-the-art nonparametric empirical Bayes methods $\hat{d}_{BG}(x)$ \citep{BrownGreenshtein2009} and $\hat{d}_{JZ}(x)$ \citep{JiangZhang2009}. The average mean squared errors across all 1,000 cells was 0.3845, 0.3599, 0.3794, 0.5881, and 0.4011, respectively. By this measure, our proposed methods were the best performers. For a more detailed comparison of the estimators, for each cell we subtracted the mean squared error of $\hat{d}_{JZ}(x)$ from the mean squared errors of the other estimators. Figure \ref{fig-data} plots the histograms of these differences and shows that our $\hat{d}_{CV}(x)$ and $\tilde{d}_h(x)$ outperformed $\hat{d}_{JZ}(x)$ on most of the cells.

\begin{figure}[t!]
  \centering
  \includegraphics[width=0.5\linewidth]{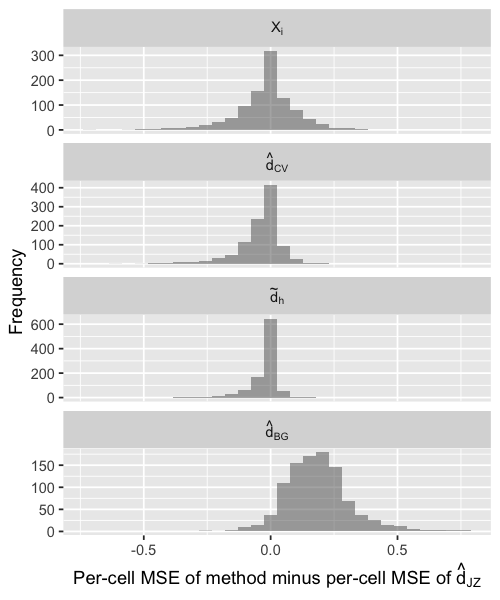}
  \caption{\label{fig-data}Histograms of the mean squared errors of different estimators minus the mean squared error of $\hat{d}_{JZ}(x)$ \citep{JiangZhang2009}, for denoising the expressions of 241 genes in each of 1,000 randomly selected cells from the mouse ileum \citep{petukhov2022cell}.}
\end{figure}

\section{Discussion}
\label{section-disc}

Nonparametric empirical Bayes approaches to simultaneous estimation are the current standard, but are philosophically challenging to describe and also face performance issues when the measurement error noise is low. We have attempted to address these issues by showing that the problem can be formulated as penalized, self-supervised, nonparametric least-squares regression. This enabled us to apply empirical risk minimization ideas to construct estimators motivated entirely by frequentist ideas that can still achieve similar levels of accuracy, in theory and in simulations, as the state-of-the-art methods. Furthermore, our estimators can outperform existing methods in low-noise settings that model more complex simultaneous estimation problems.

One potential additional benefit of our regression perspective is that it can naturally accommodate complicated decision rules, for example those that leverage covariate or structural information to estimate each $\theta_i$. This information often arises in settings where there are replications \citep{Ignatiadis2021}, results from auxiliary experiments \citep{BanerjeeMukherjeeSun2020,Zhao2021}, regular dependence structures  \citep{GreenshteinMantzuraRitov2019, Lehtinen2018}, and known heteroskedastic variances \citep{weinstein2018group,XieKouBrown2012,XieKouBrown2016,Tan2015,ZhaoBiscarri2021}. Regression modeling offers a natural way to specific, either parametrically or nonparametrically, how the additional information is believed to be related to the $\theta_i$. A multi-dimensional version of Stein's lemma can then be used derive corresponding risk estimates to be minimized. For example, given covariates $C_{i1}, \ldots, C_{ip}$ in addition to each observation $X_i$, we could posit a generalized additive model $g(X_i) + f(C_{i1}, \dots, C_{ip})$ to estimate $\theta_i$ \citep{HastieTibshirani1986, HastieFriedmanTibshirani2001}. This decomposition is closely related to the hierarchical Bayesian model proposed by \citet{ignatiadis2019covariate} under an empirical Bayes framework. For more complicated problems, modern deep learning architectures can be easily incorporated into our regression approach.

Finally, we conjecture that the performance of our estimators could be improved by incorporating additional shape constraints. For example, the nonparametric $g$-modeling estimator $\hat{d}_{JZ}(x)$ of \citet{JiangZhang2009} is, by construction, a Gaussian posterior mean, and therefore inherits a great deal of smoothness properties such as infinite, continuous, bounded derivatives. In contrast, this is not true of all functions in the two function classes we studied. It is possible that applying our framework to smaller classes of smoother functions would result in improved estimators. Related work by \citet{Zhao2021}, who applied empirical risk minimization over a function class motivated by the form of the oracle $d^\star(x)$ estimator, offers some evidence that this might be the case.

\section*{Supplementary Materials}

The supplementary material contains additional simulations and the proofs of our main and supporting results.
\par
\section*{Acknowledgements}

The authors would like to thank Sabyasachi Chatterjee for his helpful discussions.
\par

\bibliographystyle{abbrvnat}      
\bibliography{paper}             

\end{document}


\maketitle

\begin{appendix}

\section{Additional Simulations}

\subsection{Knots for the Constrained SURE Estimator, $\hat{d}(x)$}

Table \ref{table-pwl-knots} investigates knot placement for our constrained SURE estimator \eqref{eq-estimator1}, $\hat{d}(x)$, across a selection of the simulation settings from Section \ref{section-sims}. Throughout these simulations $n = 1000$ and $\sigma^2 = 1$. The Uniform setting draws $\theta_i \sim iid ~ Uniform(0,5)$; the Normal setting draws $\theta_i \sim iid ~ N(0,1)$; the Laplace setting draws $\theta_i \sim iid ~ Laplace$ with a mean of zero and a variance of one; and the ``Sparse k'' settings refer to the settings in Table \ref{table-sparse} where there $k$ of the $\theta_i = 5$ and the rest are zero. The ``M Regular'' methods set $M$ evenly spaced knots along the range of $X_1, \dots, X_n$; the ``Order Statistics'' method places knots at each observation; and the ``Percentiles'' places knots at the first, second, $\dots$, ninety-ninth percentiles of $X_1, \dots, X_n$.

\begin{table}[ht]
\centering
\begin{tabular}{l|r|r|r|r|r|r}
\hline
method & Uniform & Normal & Laplace & Sparse 5 & Sparse 50 & Sparse 500\\
\hline
1000 Regular & 647 & 514 & 503 & 38 & 79 & 175\\
\hline
300 Regular & 647 & 514 & 502 & 37 & 78 & 175\\
\hline
30 Regular & 647 & 514 & 501 & 30 & 75 & 176\\
\hline
Order Statistics & 647 & 514 & 503 & 40 & 81 & 175\\
\hline
Percentiles & 647 & 516 & 511 & 118 & 81 & 175\\
\hline
\end{tabular}
\caption{\label{table-pwl-knots}Table values are the total squared errors $\sum_{i=1}^n (\theta_i - \hat{\theta}_i)^2$ averaged across 50 replications. In each simulation $\tau_n$ is first estimated using the plug-in estimator of $\tau_n$, discussed in Section \ref{section-sims}, then $\hat{d}(x)$ is fit with each of the knot placements and the performance is measured.}
\end{table}

We see that in general the knots placement makes very little difference to the efficiency of $\hat{d}(x)$. In fact, we observe that reducing the number of knots can sometimes have a regularizing effect on the estimator. The only example in Table \ref{table-pwl-knots} where knot placement has a significant negative impact on the performance of $\hat{d}(x)$ is when percentiles are used as knots in the ``Sparse 5'' setting. Here, there are five non-zero $\theta_i$ but because the knots are placed every ten observations, $\hat{d}(x)$ cannot adapt to the very sparse means. Because fitting $\hat{d}(x)$ takes longer with more knots, we use 30 regular knots for our simulations in Section \ref{section-sims} to provide good performance and fast computation.

\subsection{Knots and Density Computation for the Monotone-Only Estimator Knots, $\tilde{d}_h(x)$}

Table \ref{table-pwc-knots} compares various knot placement and density computation methods for our monotone-only estimator \eqref{eq-estimator0}, $\tilde{d}_h(x)$. Specifically we look at the two knot placement strategies discussed in Section \ref{section-gkde}: the optimal knots, ``Opt.'' given by Theorem \ref{thrm-pwc-implementation} and the approximate knots, ``Approx.'' placed at the average of consecutive order statistics of $X_1, \dots, X_n$. Additionally, we look at the performance of the two methods of computing $\hat{f}_h(t)$: the exact method,``Exact'', and the Fast Fourier Transform, ``FFT'', approximate method. The simulation settings are exactly the same as those described before Table \ref{table-pwl-knots}.

\begin{table}[ht]
\centering
\begin{tabular}{l|l|r|r|r|r|r|r}
\hline
Knots & Density & Uniform & Normal & Laplace & Sparse 5 & Sparse 50 & Sparse 500\\
\hline
Approx. & FFT & 670 & 532 & 523 & 26 & 75 & 157\\
\hline
Approx. & Exact & 670 & 532 & 523 & 26 & 75 & 158\\
\hline
Opt. & FFT & 672 & 533 & 522 & 27 & 76 & 158\\
\hline
Opt. & Exact & 672 & 533 & 522 & 28 & 76 & 158\\
\hline
\end{tabular}
\caption{\label{table-pwc-knots}Table values are the total squared errors $\sum_{i=1}^n (\theta_i - \hat{\theta}_i)^2$ averaged across 50 replications. In each simulation, we use the asymptotic rate for bandwidth: $h = \sigma n^{-1/6}$.}
\end{table}

We observe that the approximate knots seem to slightly out-perform the optimal knots. This is slightly surprising; however, we recall that the optimal knots are chosen to be optimal for the risk estimate, not the out-of-sample risk. Furthermore, we observe that because the FFT method is so accurate, there is virtually no difference in efficiency between the two methods of computing $\hat{f}_h(t)$. Based on these observations, we recommend and use the approximate knots and FFT approximation of $\hat{f}_h(t)$, for fast computations and good simulation performance.

\subsection{Bandwidth for the Monotone-Only Estimator, $\tilde{d}_h(x)$}

Table \ref{table-pwc-bandwidths} compares various bandwidth selection methods for our monotone-only estimator \eqref{eq-estimator0}, $\tilde{d}_h(x)$. We use a fixed grid of bandwidths to study how the optimal bandwidth may vary across simulation settings and how robust $\tilde{d}_h(x)$ is to bandwidth choice. We compare our asymptotic bandwidth rate, $h = \sigma n^{-1/6}$, to common bandwidth estimation methods that are included in the \texttt{stats} package in \texttt{R} \cite{RCoreStats}. The simulation settings are exactly the same as those described before Table \ref{table-pwl-knots}.

\begin{table}[ht]
\centering
\begin{tabular}{l|r|r|r|r|r|r}
\hline
bw & Uniform & Normal & Laplace & Sparse 5 & Sparse 50 & Sparse 500\\
\hline
0.2 & 688 & 545 & 532 & 32 & 84 & 170\\
\hline
0.3 & 672 & 534 & 524 & 26 & 75 & 158\\
\hline
0.4 & 661 & 526 & 518 & 24 & 73 & 156\\
\hline
0.5 & 654 & 522 & 517 & 27 & 76 & 159\\
\hline
0.6 & 651 & 522 & 519 & 36 & 86 & 169\\
\hline
0.7 & 649 & 525 & 525 & 48 & 102 & 186\\
\hline
0.8 & 649 & 530 & 534 & 64 & 124 & 209\\
\hline
SJ & 656 & 528 & 521 & 29 & 78 & 157\\
\hline
Silverman & 661 & 532 & 523 & 30 & 79 & 170\\
\hline
UCV & 658 & 529 & 522 & 30 & 79 & 161\\
\hline
$\sigma n^{-1/6}$ & 670 & 532 & 523 & 26 & 75 & 157\\
\hline
\end{tabular}
\caption{\label{table-pwc-bandwidths}Table values are the total squared errors $\sum_{i=1}^n (\theta_i - \hat{\theta}_i)^2$ averaged across 50 replications. The ``SJ'' bandwidth is the Sheather \& Jones bandwidth \cite{SJ1991}; the ``Silverman'' bandwidth is Silverman's rule-of-thumb bandwidth \cite{Silverman1986}; the ``UCV'' is the unbiased mean integral squared error cross-validation estimate \cite{Tsybakov2009}. We use the approximate knots for each of the simulations.} 
\end{table}

Table \ref{table-pwc-bandwidths} demonstrates that all of the bandwidth selection methods we consider result in nearly optimal estimators across our simulation settings. Moreover, $\tilde{d}_h(x)$ appears to be somewhat robust to bandwidth choice as good estimators can be found over a usefully wide interval, usually around 0.2-0.3 wide. This wide interval of good bandwidths suggests that many heuristic bandwidths may result in efficient estimators. Based on these results, we use our asymptotic bandwidth rate for all of our other simulations because of its theoretical support,  computational speed, and over-all good performance.

\section{Proofs}

\begin{prop}
Let $x_1 < \dots < x_n$ and $d:\mathbb{R}\to\mathbb{R}$ be an absolutely continuous function. If $TV(d') < 2(n-1)$, then there exists an $i \in \{1, \dots, n\}$ such that either $d(x_i) \neq x_i$ or $d'(x_i) \geq 0$.
\end{prop}
\begin{proof}
We proceed by proving the contrapositive. Let $x_1 < x_2 < \dots < x_n$ be fixed and let $d:\mathbb{R}\to\mathbb{R}$ be an absolutely continuous function such that $d(x_i) = x_i$ and $d'(x_i) \leq 0$ for each $i=1, \dots, n$. Assume for contradiction that for every $t \in [x_1, x_2]$, $d'(t) < 1$. Then by our definition of $g$ and its absolute continuity we have:
\begin{align*}
x_2 - x_1
    &= d(x_2) - d(x_1) \\
    &= \int_{x_1}^{x_2} d'(t) dt \\
    &< \int_{x_1}^{x_2} 1 ~ dt \\
    &= x_2 - x_1 .
\end{align*}
Since this is a strict inequality, we arrive at a contradiction, and thus there exists $t_1 \in [x_1,x_2]$ such that $d'(t_1) \geq 1$. By analogous arguments, there exists $t_2, \dots, t_{n-1}$ such that $t_i \in [x_i, x_{i+1}]$ and $d'(t_i) \geq 1$ for $i=2, \dots, n-1$. Finally, let $\mathcal{P}$ denote the set of all partitions of $\mathbb{R}$ and $n_p$ denote the number of intervals in $\mathcal{P}$, then by the definition of supremum and our definition of $g$, we have:
\begin{align*}
TV(d') 
    &= \sup_{P\in\mathcal{P}} \sum_{i=1}^{n_p-1} \vert d'(z_{i+1}) - d'(z_i) \vert \\
    &\geq \vert d'(t_1) - d'(x_1)\vert + \vert d'(x_2) - d'(t_1)\vert + \dots \\
      &\qquad+ \vert d'(t_{n-1}) - d'(x_{n-1})\vert + \vert d'(x_n)-d'(t_{n-1})\vert \\
    &= \sum_{i=1}^{n-1} \vert d'(t_i) - d'(x_{i})\vert + \sum_{i=1}^{n-1} \vert d'(x_{i+1}) - d'(t_i)\vert \\
    &= \sum_{i=1}^{n-1} \left\{ d'(t_i) - d'(x_{i}) \right\} + \sum_{i=1}^{n-1} \left\{ d'(t_i) - d'(x_{i+1}) \right\} \\
    &= 2 \sum_{i=1}^{n-1} d'(t_i) - \left\{ d'(x_1) + d'(x_n) + 2 \sum_{i=2}^{n-1} d'(x_i) \right\} \\
    &\geq 2 \sum_{i=1}^{n-1} 1 - \left\{ 0 + 0 + 2 \sum_{i=2}^{n-1} 0 \right\} \\
    &= 2 (n-1) .
\end{align*}
\end{proof}

\begin{prop}
The optimal separable estimator \eqref{oracle}, $d^*(x)$, is bounded, monotone non-decreasing, and Lipschitz continuous; its derivative has bound total variation and is also Lipschitz continuous. Let $r(\theta) = \max_{i=1}^n \theta_i - \min_{i=1}^n \theta_i$ denote the range of $\theta_1, \dots, \theta_n$. These bounds are summarized as:
\begin{enumerate}
\item $\min_{i=1}^n \theta_i \leq d^*(x) \leq \max_{i=1}^n \theta_i$
\item $0 \leq d^{*\prime}(x) \leq \sigma^{-2} r(\theta)^2$
\item $TV(d^{*\prime}) \leq \sigma^{-2} r(\theta)^2$
\item $\left\lvert d^{*\prime\prime}(x) \right\rvert \leq \sigma^{-4} r(\theta)^3$ .
\end{enumerate}
\end{prop}
\begin{proof}
We start by showing that $d^*(x)$ is bound:
\begin{align*}
d^*(x) &= \frac{\sum_i \theta_i \phi(\theta_i-x)}{\sum_i \phi(\theta_i-x)} \\
    &\leq \frac{\sum_i \max_{i=1}^n \theta_i \phi(\theta_i-x)}{\sum_i \phi(\theta_i-x)} \\
    &= \max_{i=1}^n \theta_i
\end{align*}
and
\begin{align*}
d^*(x) &= \frac{\sum_i \theta_i \phi(\theta_i-x)}{\sum_i \phi(\theta_i-x)} \\
    &\geq \frac{\sum_i \min_{i=1}^n \theta_i \phi(\theta_i-x)}{\sum_i \phi(\theta_i-x)} \\
    &= \min_{i=1}^n \theta_i .
\end{align*}

Next, we will show that $d^*(x)$ is monotone and Lipschitz, this amounts to showing that its derivatives are non-negative and bound. Start by recalling the definition of $f_\theta(x)$ and its derivatives:
\begin{align*}
f_\theta(x) &= \frac{1}{n} \sum_{i=1}^n \frac{1}{\sigma} \phi \left(\frac{x-\theta_i}{\sigma}\right) \\
f_\theta'(x) &= \frac{1}{n} \sum_{i=1}^n \phi \left(\frac{\theta_i-x}{\sigma^2}\right) \frac{1}{\sigma} \left(\frac{x-\theta_i}{\sigma}\right) \\
f_\theta''(x) &= \frac{1}{n} \sum_{i=1}^n \left\{\frac{(\theta_i-x)^2 - \sigma^2}{\sigma^4}\right\} \frac{1}{\sigma} \phi \left(\frac{x-\theta_i}{\sigma}\right) ,
\end{align*}
and define: 
$$
w_i(x) = \frac{\frac{1}{n\sigma} \phi \left(\frac{x-\theta_i}{\sigma}\right)}{\sum_{i=1}^n \frac{1}{n\sigma} \phi \left(\frac{x-\theta_i}{\sigma}\right)} .
$$
We notice that these $w_i(x)$ are non-negative and $\sum_{i=1}^n w_i(x) = 1$ for every $x$. These $w_i(x)$ happen to be the Bayesian posterior distribution weights when our prior is $G_n(t) = n^{-1} \sum_{i=1}^n 1(\theta_i \leq x)$. Following from Tweedie's formula \eqref{Tweedie}, we can find the derivative of $d^*(x)$ in terms of $f_\theta(x)$:
\begin{align*}
\frac{d}{dx} d^*(x)
    &= \frac{d}{dx} \left\{x + \sigma^2 \frac{f'_\theta(x)}{f_\theta(x)} \right\} \\
    &= 1 + \sigma^2 \frac{f_\theta''(x)}{f_\theta(x)} - \sigma^2 \left\{\frac{f_\theta'(x)}{f_\theta(x)}\right\}^2 .
\end{align*}
Then using $w_i(x)$ we can simplify $d^{*\prime}(x)$:
\begin{align*}
d^{*\prime}(x)
    &= 1 + \sigma^2 \frac{f_\theta''(x)}{f_\theta(x)} - \sigma^2 \left\{\frac{f_\theta'(x)}{f_\theta(x)}\right\}^2 \\
    &= 1 + \sigma^2 \frac{\frac{1}{n} \sum_{i=1}^n \left\{\frac{(\theta_i-x)^2 - \sigma^2}{\sigma^4}\right\} \frac{1}{\sigma} \phi \left(\frac{x-\theta_i}{\sigma}\right)}{\frac{1}{n} \sum_{i=1}^n \frac{1}{\sigma} \phi \left(\frac{x-\theta_i}{\sigma}\right)} - \sigma^2 \left\{\frac{\frac{1}{n} \sum_{i=1}^n \left(\frac{\theta_i-x}{\sigma^2}\right) \frac{1}{\sigma} \phi \left(\frac{x-\theta_i}{\sigma}\right)}{\frac{1}{n} \sum_{i=1}^n \frac{1}{\sigma} \phi \left(\frac{x-\theta_i}{\sigma}\right)}\right\}^2 \\
    &= \sum_{i=1}^n w_i(x) \left(\frac{\theta_i-x}{\sigma}\right)^2 - \left\{\sum_{i=1}^n w_i(x) \left(\frac{\theta_i-x}{\sigma}\right)\right\}^2 \\
    &= \sum_{i=1}^n w_i(x) \left\{ \left(\frac{\theta_i-x}{\sigma}\right) - \sum_{j=1}^n w_j(x) \left(\frac{\theta_j-x}{\sigma}\right) \right\}^2 \\
    &= \sigma^{-2} \sum_{i=1}^n w_i(x) \left\{ \theta_i - \sum_{j=1}^n w_j(x) \theta_j \right\}^2 . \\
\end{align*}
Since $w_i(x)$ are positive and sum to one, we get that
\begin{align*}
d^{*\prime}(x) &\geq 0 \\
d^{*\prime}(x) 
    &\leq \sigma^{-2} \left(\max_{i=1}^n \theta_i - \min_{i=1}^n \theta_i \right)^2 \\
    &= \sigma^{-2} r(\theta)^2 ,
\end{align*}
where $r(\theta) = \max_{i=1}^n \theta_i - \min_{i=1}^n \theta_i$. 

Now, we will show that $d^{*\prime}(x)$ has finite total variation. Continuing with the definitions in the previous proof, observe that:
$$
f_\theta^{(3)}(x) = \frac{1}{n \sigma} \sum_{i=1}^n \left\{\frac{(\theta_i - x)^3 - 3 \sigma^2 (\theta_i - x)}{\sigma^6}\right\} \phi\left(\frac{x - \theta_i}{\sigma}\right).
$$
and so
\begin{align*}
d^{*\prime\prime}(x) 
    &= \sigma^2 \frac{f_\theta^{(3)}(x)}{f_\theta(x)} - 3 \sigma^2 \frac{f_\theta''(x) f_\theta'(x)}{f_\theta(x)^2} + 2 \sigma^2 \left(\frac{f_\theta'(x)}{f_\theta(x)}\right)^3 \\
    &= \sigma^{-4} \sum_{i=1}^n w_i(x) \left\{\theta_i - \sum_{j=1}^n w_j(x) \theta_j \right\}^3 .
\end{align*}
Now we can bound $TV(d^{*\prime})$:
\begin{align*}
TV(d^{*\prime})
    &= \int_{-\infty}^{\infty} \left\vert d^{*\prime\prime}(x) \right\vert dx \\
    &= \int_{-\infty}^{\infty} \left\vert \frac{1}{\sigma^4} \sum_{i=1}^n w_i(x) \left\{\theta_i - \sum_{j=1}^n w_j(x) \theta_j \right\}^3 \right\vert dx \\
    &\overset{(1)}{\leq} \int_{-\infty}^{\infty} \frac{1}{\sigma^4} \sum_{i=1}^n w_i(x) \left\vert\theta_i - \sum_{j=1}^n w_j(x) \theta_j \right\vert^3 dx \\
    &\overset{(2)}{\leq} \int_{-\infty}^{\infty} \frac{1}{\sigma^4} \sum_{i=1}^n w_i(x) \left\{\theta_i - \sum_{j=1}^n w_j(x) \theta_j \right\}^2 \left(\max_{i=1}^n \theta_i - \min_{i=1}^n \theta_i \right) dx \\
    &= \frac{\max_{i=1}^n \theta_i - \min_{i=1}^n \theta_i}{\sigma^2} \int_{-\infty}^{\infty} \frac{1}{\sigma^2} \sum_{i=1}^n w_i(x) \left\{\theta_i - \sum_{j=1}^n w_j(x) \theta_j \right\}^2 dx \\
    &= \frac{\max_{i=1}^n \theta_i - \min_{i=1}^n \theta_i}{\sigma^2} \int_{-\infty}^{\infty} {d^*}'(x) dx \\
    &\overset{(3)}{=} \frac{\max_{i=1}^n \theta_i - \min_{i=1}^n \theta_i}{\sigma^2} \int_{-\infty}^{\infty} \left\vert d^{*\prime}(x) \right\vert dx \\
    &= \frac{\max_{i=1}^n \theta_i - \min_{i=1}^n \theta_i}{\sigma^2} TV(d^*) \\
    &\overset{(4)}{\leq} \frac{\left( \max_{i=1}^n \theta_i - \min_{i=1}^n \theta_i \right)^2}{\sigma^2} \\
    &= \sigma^{-2} r(\theta)^2
\end{align*}
Here (1) is Jensen's inequality (noting that the $w_i(x)$ sum to one and are non-negative); (2) bounds a single $\theta_i - \sum_{j=1}^n \theta_j w_j(x)$ term by the range of $\theta_i$, this is easy to see by remembering that $d^*(x) = \sum_{j=1}^n \theta_j w_j(x)$ which is bound in the range of $\theta_i$; (3) is simply using the fact that $d^{*\prime}(x) \geq 0$ for every $x\in\mathbb{R}$; finally, (4) uses the bound on the total variation of $d^*(x)$ that follows from it being bounded and monotone.

Finally, we show that $d^{*\prime}(x)$ is Lipschitz by again bounding derivatives. Continuing from the previous proof:
\begin{align*}
\vert d^{*\prime\prime}(x) \vert
    &= \left\vert \sigma^{-4} \sum_{i=1}^n w_i(x) \left\{\theta_i - \sum_{j=1}^n w_j(x) \theta_j \right\}^3 \right\vert \\
    &\leq \sigma^{-4} \sum_{i=1}^n w_i(x) \left\vert \theta_i - \sum_{j=1}^n w_j(x) \theta_j \right\vert^3 \\
    &\leq \sigma^{-4} \left( \max_{i=1}^n \theta_i - \min_{i=1}^n \theta_i \right)^3 \\
    &= \sigma^{-4} r(\theta)^3 .
\end{align*}
This bound is an application of Jensen's inequality and our previous arguments using the range of $\theta_i$ to bound the summation.
\end{proof}

\begin{lemma}
Let $b_n \geq 0$ be a non-random value that may depend on $\sigma$ and $n$ and let $\mathcal{D}_{1,n}$ be defined as in \eqref{eq-D1} and $d^*(x)$ be the optimal separable estimator \eqref{oracle}, then
$$
P\left( d^* \not\in \mathcal{D}_{1,n} \right) \leq 6 n \exp\left\{\frac{- b_n^2}{2 \sigma^2}\right\} .
$$
In particular, when $b_n = (K \sigma^2 \log n)^{1/2}$ for some $K \in \mathbb{R}$, the bound becomes $6 n^{1-\frac{K}{2}}$, which can be made to converges to zero at an arbitrarily quick polynomial rate by selecting the correct $K > 2$. 
\end{lemma}
\begin{proof}
Because $X_i - \theta_i \sim N(0, \sigma^2)$ are independent we have:
\begin{align*}
P\{\min_{i} X_i - b_n &\leq \min_i \theta_i, \max_{i} X_i + b_n \leq \max_i \theta_i\} \\
    &= P\{\max_{i} (-X_i) + b_n \geq \max_i (-\theta_i), \max_{i} X_i + b_n \leq \max_i \theta_i\} \\
    &= P\{\max_{i} (-X_i) - \max_i (-\theta_i) \geq -b_n, \max_i \theta_i - \max_{i} X_i \geq b_n \} \\
    &\leq P\{\max_{i}((-X_i)-(-\theta_i)) \geq -b_n, \max_i(\theta_i-X_i) \geq b_n \} \\
    &= P\{\max_i(\theta_i-X_i) \geq -b_n, \max_i(\theta_i-X_i) \geq b_n \} \\
    &= P\{\max_i(\theta_i-X_i) \geq b_n \} \\
    &\leq n \exp\left(\frac{- b_n^2}{2 \sigma^2}\right) ,
\end{align*}
where the first inequality is a property of the max operator and the second inequality is a standard Gaussian tail bound \cite{BoucheronLugosiMassart2013}. Similarly we have
$$
P\{ \min_i \theta_i \leq \min_{i} X_i - b_n, \max_i \theta_i \leq \max_{i} X_i + b_n \} \leq n \exp\left(\frac{- b_n^2}{2 \sigma^2}\right)
$$
and 
$$
P\{ \min_i \theta_i \leq \min_{i} X_i - b_n, \max_{i} X_i + b_n \leq \max_i \theta_i \} \leq 2 n \exp\left(\frac{- b_n^2}{2 \sigma^2}\right) .
$$
The leading factor of 2 in the above result comes from the need to bound $P( \max_i \vert X_i - \theta_i \vert \geq b_n )$ rather than $P\{ \max_i (X_i - \theta_i) \geq b_n \}$. Finally, by a union bound we have:
\begin{align*}
P([\min_i \theta_i, \max_i \theta_i] &\not\subseteq [\min_i X_i - b_n, \max_i X_i + b_n] ) \\
    &\leq P(\min_{i} X_i - b_n \leq \min_i \theta_i, \max_{i} X_i + b_n \leq \max_i \theta_i ) \\
        &\qquad + P( \min_i \theta_i \leq \min_{i} X_i - b_n, \max_i \theta_i \leq \max_{i} X_i + b_n ) \\
        &\qquad + P( \min_i \theta_i \leq \min_{i} X_i - b_n, \max_{i} X_i + b_n \leq \max_i \theta_i ) \\
    &\leq 4 n \exp\left(\frac{- b_n^2}{2 \sigma^2}\right) .
\end{align*}
Now, utilizing Proposition \ref{prop-optimalproperties} and the definition of $\tau_n = \sigma^{-2}(\max_{i=1}^n X_i - \min_{i=1}^n X_i + 2 b_n)$, we have:
\begin{align*}
P\left\{ TV(d^*) > \tau_n \right\}
    &= P\left\{ \sigma^{-2} (\max_i \theta_i - \min_i \theta_i)^2 > \sigma^{-2} (\max_i X_i - \min_i X_i + 2 b_n)^2 \right\} \\
    &= P\left\{ \max_i \theta_i - \min_i \theta_i > \max_i X_i - \min_i X_i + 2 b_n \right\} \\
    &= P\left\{ (\max_i \theta_i - \max_i X_i) + (\min_i \theta_i - \min_i X_i) > 2 b_n \right\} \\
    &\leq P\left\{ \max_i (\theta_i - X_i) - \min_i (\theta_i - X_i) > 2 b_n \right\} \\
    &\leq P\left\{ \max_i \vert \theta_i - X_i \vert > b_n \right\} \\
    &\leq 2 n \exp\left(\frac{- b_n^2}{2 \sigma^2}\right) . \\
\end{align*}
By applying a final union bound, noting that $d^*(x)$ is always monotone non-decreasing, we get that:
\begin{align*}
P\left( d^* \not\in \mathcal{D}_{1,n} \right)
    &= P\left( \{[\min_i \theta_i, \max_i \theta_i] \not\subseteq [\min_i X_i - b_n, \max_i X_i + b_n]\} \cup \{TV(d^*) > \tau_n\} \right) \\
    &\leq P\left\{ [\min_i \theta_i, \max_i \theta_i] \not\subseteq [\min_i X_i - b_n, \max_i X_i + b_n]\right\} + P\left\{ TV(d^*) > \tau_n \right\} \\
    &\leq 6 n \exp\left(\frac{- b_n^2}{2 \sigma^2}\right) .
\end{align*}
Directly plugging in $b_n = (K \sigma^2 \log n)^{1/2}$ gives the bound $P\left( d^* \not\in \mathcal{D}_{1,n} \right) \leq 6 n^{1 - K/2}$ which converges to zero for any $K > 2$.
\end{proof}

\begin{thrm}
For every $g\in\mathcal{D}_{1,n}$ \eqref{eq-D1}, there exists a continuous, piecewise linear function, $\tilde{g}\in\mathcal{D}_{1,n}$ with at most $n+3$ knots, such that $\tilde{g}(X_i) = d(X_i)$ for $i = 1, \dots, n$, $\sum_{i=1}^n \tilde{g}'(X_i) = \sum_{i=1}^n g'(X_i)$, and $\tilde{g}'(x) = 0$ for $x < \min_{i=1}^n X_i$ and $x > \max_{i=1}^n X_i$.
\end{thrm}
\begin{proof}
This proof is based heavily on the proof of Theorem 1 in Boyd et al. \cite{Boyd2018}. Fix $x_1 < \dots < x_n$. Our goal is to characterize the minimizer of the risk estimate \eqref{SURE},
$$
\hat{R}_1(d) = \frac{1}{n} \sum_{i=1}^n (x_i - d(x_i))^2 + 2 \sigma^2 \frac{1}{n} \sum_{i=1}^n d'(x_i) - \sigma^2 ,
$$
over 
$$
\mathcal{D}_{1,n} = \{ d:\mathbb{R}\to\left[\min_i X_i - b_n, \max_i X_i + b_n\right] \vert d ~ \text{monotone non-decreasing}, TV(d') \leq \tau_n \} .
$$

Fix a $g\in\mathcal{D}_{n,1}$ such that $TV(g') = \tau$. Following from \cite{Boyd2018}, there is a signed measure $\mu(t)$ such that
\begin{align*}
g'(x) &= \int_{-\infty}^{\infty} 1(t \leq x) d\mu(t) \\
g(x) &= g(0) + \int_{-\infty}^{\infty} (x-t)_+ d\mu(t) .
\end{align*}

First notice that our risk estimate $\hat{R}_1(d)$ only depends on the function's behavior at $\{x_1, \dots, x_n\}$. This means that we can restrict $\mu(t)$ to $[x_1,x_n]$ without loss of generality, as changes made outside of that range are not relevant to our risk estimate.

Further notice that since $g(x)$ is monotone non-decreasing, it is either constant for large $\vert x \vert$ or asymptotically approaches a constant as $x \to \pm \infty$. In light of the prior observation, we can assume $g(x)$ is constant outside of $[x_1, x_n]$ without loss of generality. This implies that $\int_{-\infty}^{\infty} d\mu(t) = \int_{x_1}^{x_n} d\mu(t) = 0$. We notice that this argument generalizes to any $g(x)$ with finite total variation. 

Finally notice that $g(0)$ affects the calculation of $\hat{R}_1(g)$ but does not depend on $\mu(t)$ so we will ignore it (set it to zero) for the rest of this proof. This is equivalent to absorbing it into the observations in the squared error term or simply minimizing it out of the problem. 

Our goal now is to show that there is a discrete, signed measure $\omega(t)$ such that the corresponding function: $\tilde{g}(x) = \int_{x_1}^{x_n} (x-t)_+ d\omega(t)$ matches $g(x)$ on each $x_i$ and has the same penalty value. We will further ensure that the function is constant outside of $[x_1, x_n]$ and has a derivative with the same total variation as $g'(x)$. Observe that
\begin{align*}
g(x_i) 
    &= \int_{x_1}^{x_n} (x_i-t)_+ d\mu(t) \\
\frac{1}{n} \sum_{i=1}^n g'(x_i)
    &= \frac{1}{n} \sum_{i=1}^n \int_{x_1}^{x_n} 1(x_i \geq t) d\mu(t) \\
    &= \int_{x_1}^{x_n} \left\{ \frac{1}{n} \sum_{i=1}^n 1(x_i \geq t) \right\} d\mu(t) \\
0
    &= \int_{x_1}^{x_n} d\mu(t) . \\
\end{align*}
Now, define the vector $v$ as
\begin{align*}
v 
    &= \left( g(x_1), \dots, g(x_n), \frac{1}{n} \sum_{i=1}^n g'(x_i), 0 \right) \\
    &= \left( \int_{x_1}^{x_n} (x_1-t)_+ d\mu(t), \dots, \int_{x_1}^{x_n} (x_n-t)_+ d\mu(t), \right. \\
        &\qquad\qquad \left. \int_{x_1}^{x_n} \left\{ \frac{1}{n} \sum_{i=1}^n 1(x_i \geq t) \right\} d\mu(t), \int_{x_1}^{x_n} d\mu(t) \right) \\
    &= \int_{x_1}^{x_n} \left( (x_1-t)_+, \dots, (x_n-t)_+, \frac{1}{n} \sum_{i=1}^n 1(x_i \geq t), 1 \right) d\mu(t) . \\
\end{align*}
Notice that $n^{-1} \sum_{i=1}^n 1(x_i \geq t) \in [0,1]$ for every $t\in[x_1,x_n]$. Now define the convex set $C \subset \mathbb{R}^{n+2}$ as:
$$
C = \left\{ \pm \left( \tau (x_1 - t)_+, \dots, \tau (x_n - t)_+, \tau \frac{1}{n} \sum_{i=1}^n 1(x_i \geq t), \tau \right)\ :\ t\in[x_1, x_n] \right\} .
$$
We do not need $C$ to be convex, so long as we can say that is in the convex hull of $C$; however, we can see that since each components is convex $C$ is an intersection of finitely many convex sets, and so it is convex. Notice that because $TV(g') = TV(\mu) = \int_{-\infty}^{\infty} d\vert\mu\vert(t) = \tau$, we have $v \in C \subset \mathbb{R}^{n+2}$. Since $C$ is convex, $v$ is also in the convex hull of $C$, ie $v\in\text{conv}(C)$.

Now by Caratheodory’s theorem for convex hulls \cite{Caratheodory1911} we can represent $v$ as the linear combination of at most $n+3$ points from $\text{conv}(C)$. The standard statement of Caratheodory’s theorem for convex hulls ensures that this linear combination is a convex combination, ie the coefficients satisfy $\alpha_j \geq 0$ and $\sum_{j=1}^{n+3} \alpha_j = 1$. It will be convenient to recognize that this implies there is at least one linear combination where the coefficients simply satisfy $\sum_{j=1}^M \vert \alpha_j \vert = 1$. Observe that there are more of these linear combinations than the convex combinations due to the loosening of the non-negativity constraint. Denote the points ensured by our extension of the Caratheodory’s theorem for convex hulls with their indices $t_1,\dots, t_{n+3} \in [x_1, x_n]$ and assume the corresponding coefficients satisfy $\sum_{j=1}^{n+3} \vert\alpha_j\vert = 1$. Using this representation we have a new signed measure defined as: 
$$\omega(x) = \sum_{j=1}^{n+3} \omega_j \delta_{t_j}(x) , $$
where $\delta_{t_j}$ is a Kronecker delta function placing a mass of one at $t_j$ and $\omega_j = \tau \alpha_j$. Finally, we have:
\begin{align*}
g(x_i)
    &= \sum_{j=1}^{n+3} \alpha_j \tau (x_i-t_j)_+ \\
    &= \sum_{j=1}^{n+3} \omega_j (x_i-t_j)_+ \\
    &= \tilde{g}(x_i) \\
\frac{1}{n} \sum_{i=1}^n g'(x_i)
    &= \sum_{j=1}^{n+3} \alpha_j \tau \frac{1}{n} \sum_{i=1}^n 1(x_i \geq t_j) \\
    &= \frac{1}{n} \sum_{i=1}^n \sum_{j=1}^{n+3} \omega_j 1(x_i \geq t_j) \\
    &= \frac{1}{n} \sum_{i=1}^n \tilde{g}'(x_i) \\
0
    &= \sum_{j=1}^{n+3} \alpha_j \tau \\
    &= \sum_{j=1}^{n+3} \omega_j \\
\end{align*}
Further notice that $\sum_{j=1}^{n+3} \vert \omega_j \vert = \sum_{j=1}^{n+3} \vert \alpha_j \tau \vert = \tau \sum_{j=1}^{n+3} \vert \alpha_j \vert = \tau$ so that $TV(\tilde{g}') = TV(g')$.

This means that we have found a new, discrete, signed measure $\omega$ such that the corresponding function $\tilde{g}(x) = \sum_{j=1}^{n+3} \omega_j (x - t_j)_+$ matches $g$ on each $x_1, \dots, x_n$, has the same penalty value, and satisfies $TV(\tilde{g}') = TV(g')$. Further 1) since $\sum_{j=1}^{n+3} \omega_j = 0$, we have that $\tilde{g}$ saturates, and 2) because $\tilde{g}(x_i) = g(x_i)$, if $g$ is monotone non-decreasing, $\tilde{g}$ must be, too. Thus we can assume the minimizer of $\hat{R}_1(d)$ over $\mathcal{D}_{n,1}$ is a piecewise linear function with at most $n+3$ knots without loss of generality. 

We notice that similar to Proposition \ref{prop-abscont}, this argument requires $2 n + 2$ knots if we want to match specific derivatives at each $x_i$, the reduction in required knots is because we are only matching the total penalty rather than each term in the summation. 
\end{proof}

\begin{extralemma}
\label{lemma-estimation-one}
Let $X_i \sim N(\theta_i, \sigma^2)$ for $i = 1, \dots, n$ be independent and $\sigma > 0$ known, then
$$
E \left\vert \frac{1}{n}\sum_{i=1}^n (X_i - \theta_i)^2 - \sigma^2 \right\vert \leq \sigma^2 (2 / n)^{1/2} .
$$
\end{extralemma}
\begin{proof}
First observe that by independence and normality of $X_i$, $\sigma^{-2} (X_i - \theta_i)^2$ are independent $\chi^2_1$ random variables, so
\begin{align*}
E \left\{ \frac{1}{n}\sum_{i=1}^n (X_i - \theta_i)^2 - \sigma^2 \right\}^2
    &= \frac{\sigma^4}{n^2} E \left\{\sum_{i=1}^n \left(\frac{X_i - \theta_i}{\sigma}\right)^2 - n\right\}^2 \\
    &= \frac{\sigma^4}{n^2} \text{Var} \left\{\sum_{i=1}^n \left(\frac{X_i - \theta_i}{\sigma}\right)^2 \right\} \\
    &= \frac{\sigma^4}{n^2} \cdot 2n \\
    &= \frac{2 \sigma^4}{n} .
\end{align*}
Our result now follows from Jensen's inequality: 
\begin{align*}
E \left\vert \frac{1}{n}\sum_{i=1}^n (X_i - \theta_i)^2 - \sigma^2 \right\vert 
    &\leq \left[ E \left\{ \frac{1}{n}\sum_{i=1}^n (X_i - \theta_i)^2 - \sigma^2 \right\}^2 \right]^{1/2} \\
    &= \sigma^2 (2 / n)^{1/2} .
\end{align*}
\end{proof}

\begin{extralemma}
\label{lemma-estimation-two}
Let $X_i \sim N(\theta_i, \sigma^2)$ for $i = 1, \dots, n$ be independent with $\vert \theta_i \vert \leq C_n$ where $\sigma > 0$ and $C_n$ are known, then
$$
E \left\vert \frac{1}{n} \sum_{i=1}^n (X_i-\theta_i) \theta_i \right\vert \leq \sigma C_n n^{-1/2}.
$$
\end{extralemma}
\begin{proof}
First observe that for all $i = 1, \dots, n$, $(X_i - \theta_i) \theta_i$ are independent, mean zero random variables, then
\begin{align*}
E \left\{\frac{1}{n}\sum_{i=1}^n (X_i-\theta_i)\theta_i \right\}^2
    &= \frac{1}{n^2} \sum_{i=1}^n \sum_{j=1}^n E \{ (X_i-\theta_i) \theta_i \cdot (X_j-\theta_j) \theta_j \} \\
    &= \frac{1}{n^2} \sum_{i=1}^n E \{ (X_i-\theta_i) \theta_i \}^2 \\
    &= \frac{1}{n^2} \sum_{i=1}^n \sigma^2 \theta_i^2 \\
    &\leq \frac{\sigma^2}{n^2} \sum_{i=1}^n C_n^2 \\
    &= \frac{\sigma^2  C_n^2}{n} .
\end{align*}
Our result now follows from Jensen's inequality: 
\begin{align*}
E \left\vert \frac{1}{n}\sum_{i=1}^n (X_i-\theta_i)\theta_i \right\vert
    &\leq \left[ E \left\{ \frac{1}{n}\sum_{i=1}^n (X_i-\theta_i)\theta_i \right\}^2 \right]^{1/2} \\
    &\leq \left( \frac{\sigma^2  C_n^2}{n} \right)^{1/2} \\
    &= \sigma C_n n^{-1/2} .
\end{align*}
\end{proof}

\begin{extralemma}
\label{lemma-estimation-rademacher}
Let $\mathcal{F}$ denote a uniformly bound, separable function class and fix $x_i \in \mathbb{R}$ for $i=1, \dots, n$. Define the mean-zero random variable $Z_f = n^{-1/2} \sum_{i=1}^n \epsilon_i f(x_i)$ where $\epsilon_i$ are independent, identically distributed Rademacher random variables. Also define the empirical $L_2$ norm as $\Vert f \Vert_{L_n}^2 = n^{-1} \sum_{i=1}^n f(x_i)^2$. Then
$$
E \sup_{f\in\mathcal{F}} \left\vert Z_f \right\vert \leq C \int_{0}^{2 b} \{ \log N(\mathcal{F}; L_n; t) \}^{1/2} dt , 
$$
where $C > 0$ is a universal constant and $\log N(\mathcal{F}; L_n; t)$ denotes the metric entropy of the function class $\mathcal{F}$ with respect to the metric $L_n$.
\end{extralemma}
\begin{proof}
The proof follows from the observation that $Z_f$ has sub-Gaussian increments with parameter $\Vert f - g \Vert_{L_n}$. Further, we observe that if $\mathcal{F}$ is uniformly bound by $b$, then $\sup_{f,g\in\mathcal{F}} \Vert f - g \Vert_{L_n} \leq 2 b$. Now we apply a chaining argument and Dudley's integral bound to get our desired result. More details, definitions and a complete proof can be found in \cite{VandervaartWellner1996, Wainwright2019}.
\end{proof}

\begin{extralemma}
\label{lemma-approximation-expectation-bounds}
Let  $R_n(d; \theta) = n^{-1} \sum_{i=1}^n \{\theta_i - d(X_i)\}^2$ and $R(d; \theta) = E R_n(d; \theta)$. Let $\mathcal{D}_n$ denote one of $\mathcal{D}_{1,n}$ \eqref{eq-D1}, $\mathcal{D}_{0,n}$ \eqref{eq-D0}, or $\mathcal{C}_{0,n}$ (from Corollary \ref{corollary-pwc-rate}), we note these function classes have the same data-dependent range in terms of $X_i$ and $b_n$. Let $d^*(x)$ be the optimal separable estimator \eqref{oracle} and $\hat{d}^* = \arg\min_{d \in\mathcal{D}_n} R_n(d; \theta)$. Then for $n \geq 3$,
\begin{align*}
E \left\{ R_n(\hat{d}^*; \theta) \right.
  &\left. \vphantom{\hat{d}^*} - R_n(d^*; \theta) \mid d^* \not\in\mathcal{D}_n \right\} \\
  &\leq 5 \sigma^2 \log(n) + 4 \sigma (4 C_n + b_n) \log^{1/2}(n) + (b_n^2 + 8 b_n C_n + 12 C_n^2) .
\end{align*}
\end{extralemma}
\begin{proof}
Recall that $\max_{i=1}^n \vert \theta_i \vert \leq C_n$. First use the triangle inequality and the range bounds on $d^*(x)$ from Proposition \ref{prop-optimalproperties} and $\mathcal{D}_n$ to see that:
\begin{align*}
\max_{i=1}^n \left\vert d^*(X_i) - \hat{d}^*_n(X_i) \right\vert
    &\leq \max_{i=1}^n \left\vert d^*(X_i)\right\vert + \max_{i=1}^n \left\vert \hat{d}^*(X_i) \right\vert \\
    &\leq \max_{i=1}^n \vert \theta_i\vert + \left(\max_{i=1}^n \vert X_i\vert + b_n \right) \\
    &\leq \max_{i=1}^n \vert \theta_i\vert + \left(\max_{i=1}^n \vert X_i - \theta_i\vert + \max_{i=1}^n \vert \theta_i\vert + b_n \right) \\
    &= \max_{i=1}^n \vert X_i - \theta_i\vert + 2 \max_{i=1}^n \vert \theta_i\vert + b_n \\
    &\leq \max_{i=1}^n \vert X_i - \theta_i\vert + 2 C_n + b_n .
\end{align*}
Now, we can bound the difference in oracle loss as:
\begin{align*}
R_n(\hat{d}^*; \theta) &- R_n(d^*; \theta) \\
    &= \frac{1}{n} \sum_{i=1}^n \left\{\theta_i - \hat{d}^*(X_i) \right\}^2 - \frac{1}{n} \sum_{i=1}^n \{\theta_i - d^*(X_i)\}^2 \\
    &= \frac{1}{n} \sum_{i=1}^n \left[ \left\{ d^*(X_i) - \hat{d}^*(X_i) \right\}^2 + 2 (\theta_i - d^*(X_i))\left\{d^*(X_i) - \hat{d}^*(X_i) \right\} \right] \\
    &\leq \frac{1}{n} \sum_{i=1}^n \left[ \left\{ d^*(X_i) - \hat{d}^*(X_i) \right\}^2 + 2 \left\vert d^*(X_i) - \hat{d}^*(X_i) \right\vert \cdot \max_{i=1}^n \left\vert \theta_i - d^*(X_i)\right\vert \right] \\
    &\leq \frac{1}{n} \sum_{i=1}^n \left[ \left\{ d^*(X_i) - \hat{d}^*(X_i) \right\}^2 + 2 \left\vert d^*(X_i) - \hat{d}^*(X_i) \right\vert \left(C_n + C_n \right) \right] \\
    &= \frac{1}{n} \sum_{i=1}^n \left[ \left\{ \left\vert d^*(X_i) - \hat{d}^*(X_i) \right\vert \right\}^2 + 4 C_n \left\vert d^*(X_i) - \hat{d}^*(X_i) \right\vert \right] \\
    &\leq \frac{1}{n} \sum_{i=1}^n \left[ \left\{ \max_{i=1}^n \left\vert d^*(X_i) - \hat{d}^*(X_i) \right\vert \right\}^2 + 4 C_n \max_{i=1}^n \left\vert d^*(X_i) - \hat{d}^*(X_i) \right\vert \right] \\
    &= \left\{ \max_{i=1}^n \left\vert d^*(X_i) - \hat{d}^*(X_i) \right\vert \right\}^2 + 4 C_n \max_{i=1}^n \left\vert d^*(X_i) - \hat{d}^*(X_i) \right\vert .
\end{align*}
The first inequality in the above bound is a Holder 1-infinity bound; the second inequality is another instance of the triangle inequality and range bounds; and, the third inequality uses the fact that the max is larger than the average. Finally, we can plug these results into the main bound we are interested in. Below we use maximal inequalities for both Gaussian random variables and chi-squared random variables \cite{BoucheronLugosiMassart2013}. Let $n \geq 3$, then
\begin{align*}
E \{ R_n(\theta, \hat{d}^*_n)
  &- R_n(\theta, d^*) \mid d^* \not\in \mathcal{D}_n \} \\
    &\leq E \left[ \left\{ \max_{i=1}^n \left\vert d^*(X_i) - \hat{d}^*(X_i) \right\vert \right\}^2 + 4 C_n \max_{i=1}^n \left\vert d^*(X_i) - \hat{d}^*(X_i) \right\vert \mid \mathcal{D}\not\subseteq\mathcal{D}_n \right] \\
    &\leq E \left\{ \left( \max_{i=1}^n \vert X_i - \theta_i\vert + 2 C_n + b_n \right)^2 + 4 C_n \left( \max_{i=1}^n \vert X_i - \theta_i\vert + 2 C_n + b_n \right) \right\} \\
    &= E \left\{ \left( \max_{i=1}^n \vert X_i - \theta_i\vert \right)^2 + (8 C_n + 2 b_n) \max_{i=1}^n \vert X_i - \theta_i\vert \right. \\
        &\qquad\quad \left. \vphantom{\left( \max_{i=1}^n \vert X_i - \theta_i\vert \right)^2} + (b_n^2 + 8 b_n C_n + 12 C_n^2) \right\} \\
    &\leq \sigma^2 \left\{ 1 + 2 \{\log(n)\}^{1/2} + 2 \log(n) \right\} + (8 C_n + 2 b_n) \sigma \{2 \log(2 n)\}^{1/2} \\
        &\qquad\qquad+ (b_n^2 + 8 b_n C_n + 12 C_n^2) \\
    &\leq 5 \sigma^2 \log(n) + 4 \sigma (4 C_n + b_n) (\log n)^{1/2} + (b_n^2 + 8 b_n C_n + 12 C_n^2) . \\
\end{align*}
\end{proof}

\begin{thrm}
Let $X_i \sim N(\theta_i, \sigma^2)$ for $i=1,\dots,n$ be independent such that $\sigma^2 > 0$ is known and assume $\max_{i=1}^n \vert \theta_i \vert < C_n$. Further, let $b_n = \sigma (3 \log n)^{1/2}$ and $\tau_n = \sigma^{-2} (\max_{i=1}^n X_i - \min_{i=1}^n X_i + 2 b_n)^2$. Define $\hat{d}(x)$ as in \eqref{eq-estimator1} and let $d^*(x)$ be the optimal separable estimator \eqref{oracle}, then
$$
R(\hat{d}; \theta) - R(d^*; \theta) = \mathcal{O}\left\{ n^{-1/2} \left(C_n + \log^{1/2} n \right)^2 \right\} .
$$
\end{thrm}
\begin{proof}
Recall or define the following estimators: 
\begin{align*}
d^*(x) &= \arg\min_{d:\mathbb{R}\to\mathbb{R}} R(d; \theta) \\
\hat{d}^*(x) &= \arg\min_{d\in\mathcal{D}_{1,n}} R_n(d; \theta) \\
\hat{d}(x) &= \arg\min_{d\in\mathcal{D}_{1,n}} \hat{R}_1(d) .
\end{align*}
We break the excess risk into two terms: 
$$
R(\hat{d}; \theta) - R(d^*; \theta) = \left\{ R(\hat{d}; \theta) - R(\hat{d}^*; \theta) \right\} + \left\{ R(\hat{d}^*; \theta) - R(d^*; \theta) \right\} .
$$
We refer to these terms as estimation error and an approximation error, respectively. We start by bounding the approximation error.

First observe that we can bound the approximation error using the law of total expectation (tower rule), Lemma \ref{lemma-prob}, and Lemma \ref{lemma-approximation-expectation-bounds}:
\begin{align*}
R(\hat{d}^*; \theta) - R(d^*; \theta)
    &= E R_n(\hat{d}^*; \theta) - E R_n(d^*; \theta) \\
    &= E \left\{ R_n(\theta, \hat{\delta}^*) - R_n(\theta, \delta^*) \vert d^* \in \mathcal{D}_{1,n} \right\} P\left\{ d^* \in \mathcal{D}_{1,n} \right\} \\
        &\qquad+ E \left\{ R_n(\theta, \hat{\delta}^*) - R_n(\theta, \delta^*) \vert d^* \not\in \mathcal{D}_{1,n} \right\} P\left\{ d^* \not\in \mathcal{D}_{1,n} \right\} \\
    &\leq 0 \cdot P\left\{ d^* \in \mathcal{D}_{1,n} \right\} \\
        &\qquad\qquad+ E \left\{ R_n(\theta, \hat{\delta}^*) - R_n(\theta, \delta^*) \vert d^* \not\in \mathcal{D}_{1,n} \right\} P\left\{ d^* \not\in \mathcal{D}_{1,n} \right\} \\
    &= E \left\{ R_n(\theta, \hat{\delta}^*) - R_n(\theta, \delta^*) \vert d^* \not\in \mathcal{D}_{1,n} \right\} P\left\{ d^* \not\in \mathcal{D}_{1,n} \right\} \\
    &\leq \left\{ 5 \sigma^2 \log(n) + 4 \sigma (4 C_n + b_n) \log^{1/2}(n) + (b_n^2 + 8 b_n C_n + 12 C_n^2) \right\} \\
        &\qquad\qquad \cdot 6 n \exp\left(\frac{- b_n^2}{2 \sigma^2}\right) 
\end{align*}
Now, letting $b_n = \sigma (K \log n)^{1/2}$, we get:
$$
R(\hat{d}^*; \theta) - R(d^*; \theta) = \mathcal{O}\left\{ n^{1-\frac{K}{2}} \left(C_n + \sigma \log^{1/2} n \right)^2 \right\}
$$
We will plug-in $K = 3$ later to minimize the total excess risk bound.

Now we bound the estimation error using M-estimation techniques (symmetrization and metric entropy bounds) \cite{VandervaartWellner1996, Wainwright2019}. When necessary, we will use subscripts to denote which random variables an expectation is being taken over. Let $X_i'$ be an independent copy of $X_i$ and $\epsilon_i$ be independent Rademacher random variables. Then observe by a symmetrization argument and Theorem 2.2 in \cite{Koltchinskii2011}: 
\begin{align*}
E \sup_{d\in\mathcal{D}_n}
    & \left\vert \frac{1}{n}\sum_{i=1}^n \{(X_i - \theta_i) d(X_i) - E_{X_i} (X_i - \theta_i) d(X_i) \} \right\vert \\
    &= E_X \sup_{d\in\mathcal{D}_n} \left\vert \frac{1}{n}\sum_{i=1}^n \left\{(X_i - \theta_i) d(X_i) - E_{X_i'} (X_i' - \theta_i) d(X_i') \right\} \right\vert \\
    &= E_X \sup_{d\in\mathcal{D}_n} \left\vert E_{X' \vert X} \frac{1}{n}\sum_{i=1}^n \left\{(X_i - \theta_i) d(X_i) - (X_i' - \theta_i) d(X_i') \right\} \right\vert \\
    &\leq E_{X X'} \sup_{d\in\mathcal{D}_n} \left\vert \frac{1}{n}\sum_{i=1}^n \left\{(X_i - \theta_i) d(X_i) - (X_i' - \theta_i) d(X_i') \right\} \right\vert \\
    &= E_{X X' \epsilon} \sup_{d\in\mathcal{D}_n} \left\vert \frac{1}{n}\sum_{i=1}^n \epsilon_i \left\{(X_i - \theta_i) d(X_i) - (X_i' - \theta_i) d(X_i') \right\} \right\vert \\
    &\leq 2 E_{X \epsilon} \sup_{d\in\mathcal{D}_n} \left\vert \frac{1}{n}\sum_{i=1}^n \epsilon_i (X_i - \theta_i) d(X_i) \right\vert \\
    &= 2 E_{X \epsilon} \sup_{d\in\mathcal{D}_n} \left\vert \frac{1}{n}\sum_{i=1}^n \epsilon_i \vert X_i - \theta_i \vert d(X_i) \right\vert \\
    &= 2 E_{X \epsilon} \sup_{d\in\mathcal{D}_n} \left\vert \frac{1}{n}\sum_{i=1}^n \epsilon_i \vert X_i - \theta_i \vert d(X_i) \frac{\max_{i=1}^n \vert X_i - \theta_i\vert}{\max_{i=1}^n \vert X_i - \theta_i\vert} \right\vert \\
    &= 2 E_{X \epsilon} \left\{ \max_{i=1}^n \vert X_i - \theta_i\vert \cdot \sup_{d\in\mathcal{D}_n} \left\vert \frac{1}{n}\sum_{i=1}^n \epsilon_i \frac{\vert X_i - \theta_i\vert}{\max_{i=1}^n \vert X_i - \theta_i\vert} d(X_i) \right\vert \right\} \\
    &= 2 E_{X} \left[ \max_{i=1}^n \vert X_i - \theta_i\vert \cdot E_{\epsilon \vert X} \left\{ \sup_{d\in\mathcal{D}_n} \left\vert \frac{1}{n}\sum_{i=1}^n \epsilon_i \frac{\vert X_i - \theta_i\vert}{\max_{i=1}^n \vert X_i - \theta_i\vert} d(X_i) \right\vert \right\} \right] \\
    &\leq 2 E_{X} \left[ \max_{i=1}^n \vert X_i - \theta_i\vert\cdot E_{\epsilon \vert X} \left\{ \sup_{d\in\mathcal{D}_n} \left\vert \frac{1}{n}\sum_{i=1}^n \epsilon_i d(X_i) \right\vert \right\} \right] .
\end{align*}
Now we can bound the inner expectation using Lemma \ref{lemma-estimation-rademacher}. Van de Geer \cite{vandeGeer2000} uses \cite{BirmanSolomjak1967} and \cite{vandeGeer1990} to establish that
$$
\log N(\mathcal{D}_{1,n}; L_n; t) \leq A \frac{B_n}{t} ,
$$
for some universal constant $A > 0$ and $B_n = \max_{i=1}^n X_i - \min_{i=1}^n X_i + 2 b_n$. The numerator of the entropy bound comes from scaling the class of all monotone functions whose range is $[0,1]$ to have the range of $\mathcal{D}_{1,n}$. Notice that $B_n \leq 2 \max_{i=1}^n \vert X_i - \theta_i \vert + 2 C_n + 2 b_n$, so applying Lemma \ref{lemma-estimation-rademacher} and standard maximal inequalities \cite{BoucheronLugosiMassart2013}, for $n \geq 3$ we have:
\begin{align*}
E \sup_{d\in\mathcal{D}_n} 
    &\left\vert \frac{1}{n}\sum_{i=1}^n \{(X_i - \theta_i) d(X_i) - E_{X_i} (X_i - \theta_i) d(X_i) \} \right\vert \\
    &\leq 2 E_{X} \left[ \max_{i=1}^n \vert X_i - \theta_i\vert\cdot E_{\epsilon \vert X} \left\{ \sup_{d\in\mathcal{D}_n} \left\vert \frac{1}{n}\sum_{i=1}^n \epsilon_i d(X_i) \right\vert \right\} \right] \\
    &\leq 2 E_{X} \left[ \max_{i=1}^n \vert X_i - \theta_i\vert \cdot n^{-1/2} C \int_{0}^{2 B_n} \{ \log N(\mathcal{F}; L_n; t) \}^{1/2} dt \right] \\
    &\leq 2 E_{X} \left\{ \max_{i=1}^n \vert X_i - \theta_i\vert \cdot n^{-1/2} C \int_{0}^{2 B_n} ( A B_n / t )^{1/2} dt \right\} \\
    &= 2 E_{X} \left\{ \max_{i=1}^n \vert X_i - \theta_i\vert \cdot n^{-1/2} 2 \sqrt{2} C \sqrt{A} B_n \right\} \\
    &\leq \frac{4 C (2 A)^{1/2}}{n^{1/2}} E_{X} \left\{ \max_{i=1}^n \vert X_i - \theta_i\vert \cdot 2 \left(\max_{i=1}^n \vert X_i - \theta_i \vert + C_n + b_n\right) \right\} \\
    &= C_1 n^{-1/2} E_{X} \left\{ \left( \max_{i=1}^n \vert X_i - \theta_i\vert\right)^2 +  (C_n + b_n) \max_{i=1}^n \vert X_i - \theta_i \vert \right\} \\
    &\leq C_1 n^{-1/2} E_{X} \left\{  \max_{i=1}^n (X_i - \theta_i)^2 +  (C_n + b_n) \max_{i=1}^n \vert X_i - \theta_i \vert \right\} \\
    &\leq C_1 n^{-1/2} \left[ \sigma^2 \left\{ 1 + 2 (\log n)^{1/2} + 2 \log(n) \right\} + (C_n + b_n) \sigma \{2 \log(2 n)\}^{1/2} \right] \\
    &\leq C_1 n^{-1/2} \left\{ 5 \sigma^2 \log n + 2 \sigma (C_n + b_n) (\log n)^{1/2} \right\} \\
    &= \mathcal{O}\left[ n^{-1/2} \left\{ \log n + C_n (\log n)^{1/2} \right\} \right] ,
\end{align*}
where $C_1 = 4 C (2 A)^{1/2} > 0$ is a constant and the final big-O notation is simplified by plugging in $b_n = \sigma ( K \log n )^{1/2}$ with $K > 0$.

A similar bound, but with simpler symmetrization, can be used to bound the final component of our proof. First, we need to make a few observations about the class of derivatives in $\mathcal{D}_{1,n}$. Fix an arbitrary $d\in\mathcal{D}_{1,n}$, then because $d$ is bounded and monotone, we have $d'(x) \to 0$ as $x \to \pm\infty$. Since $TV(d') \leq \tau_n = \sigma^{-2} B_n^2$ we know that the class of derivatives is bound in $[-0.5 \tau_n, 0.5 \tau_n]$, hence the function class, $\{ d' \vert d\in\mathcal{D}_{1,n}\}$ has a range uniformly bound by $\tau_n/2$ and a total variation less than $\tau_n$. This implies the following metric entropy bound on the class of derivatives \cite{vandeGeer2000}:
$$
\log N(\{ d' \vert d\in\mathcal{D}_{1,n}\}; L_n; t) \leq A_2 \frac{\tau_n}{t}, 
$$
where $A_2 > 0$ is fixed. Recall that $\tau_n = \sigma^{-2} (\max_{i=1}^n X_i - \min_{i=1}^n X_i + 2 b_n)^2$ and let $n \geq 3$, then by symmetrization, a chaining bound (Lemma \ref{lemma-estimation-rademacher}), and standard maximal inequalities \cite{BoucheronLugosiMassart2013}, we have: 
\begin{align*}
E_X \sup_{d\in\mathcal{D}_n} 
    & \left\vert \frac{1}{n}\sum_{i=1}^n \{ d'(X_i) - E_{X_i} d'(X_i) \} \right\vert \\
    &= E_X \sup_{d\in\mathcal{D}_n} \left\vert \frac{1}{n}\sum_{i=1}^n \left\{ d'(X_i) - E_{X_i'} d'(X_i') \right\} \right\vert \\
    &= E_X \sup_{d\in\mathcal{D}_n} \left\vert E_{X' \vert X} \frac{1}{n}\sum_{i=1}^n \left\{ d'(X_i) - d'(X_i') \right\} \right\vert \\
    &\leq E_{X X'} \sup_{d\in\mathcal{D}_n} \left\vert \frac{1}{n}\sum_{i=1}^n \left\{ d'(X_i) - d'(X_i') \right\} \right\vert \\
    &= E_{X X' \epsilon} \sup_{d\in\mathcal{D}_n} \left\vert \frac{1}{n}\sum_{i=1}^n \epsilon_i \left\{ d'(X_i) - d'(X_i') \right\} \right\vert \\
    &\leq 2 E_{X \epsilon} \sup_{d\in\mathcal{D}_n} \left\vert \frac{1}{n}\sum_{i=1}^n \epsilon_i d'(X_i) \right\vert \\
    &\leq 2 E_X \left[ n^{-1/2} C \int_{0}^{\tau_n} \{ \log N(\{ d' \vert d\in\mathcal{D}_{1,n}\}; L_n; t) \}^{1/2} dt \right] \\
    &\leq 2 E_X \left\{ n^{-1/2} C \int_{0}^{\tau_n} ( A_2 \tau_n / t )^{1/2} dt \right\} \\
    &= 4 C A_2^{1/2} n^{-1/2} E_X (\tau_n ) \\
    &\leq 16 C A_2^{1/2} \sigma^{-2} n^{-1/2} E_X \left\{ \left( \max_{i=1}^n \vert X_i - \theta_i \vert + C_n + b_n \right)^2 \right\} \\
    &= C_2 n^{-1/2} E_X \left\{ \max_{i=1}^n (X_i - \theta_i)^2 + 2 (C_n + b_n) \max_{i=1}^n \vert X_i - \theta_i \vert + (C_n + b_n)^2 \right\} \\
    &\leq C_2 n^{-1/2} \left[ \sigma^2 \left\{1 + 2 (\log n)^{1/2} + 2 \log(n) \right\} \right. \\
        &\qquad \left. \vphantom{(\log n)^{1/2}} + 2 (C_n + b_n) \sigma \{2 \log(2 n)\}^{1/2} + (C_n + b_n)^2 \right] \\
    &\leq C_2 n^{-1/2} \left[ 5 \sigma^2 \log(n) + 4 (C_n + b_n) \sigma (\log n)^{1/2} + (C_n + b_n)^2 \right] \\
    &= n^{-1/2} \left\{ (C_n + b_n) + \sigma (\log n)^{1/2} \right\}^2 \\
    &= \mathcal{O}\left[ n^{-1/2} \left\{ C_n + \sigma (\log n)^{1/2} \right\}^2 \right] , 
\end{align*}
where $C_2 = 16 C A_2^{1/2} \sigma^{-2} > 0$ is a constant and the final big-O notation is simplified by plugging in $b_n = \sigma ( K \log n )^{1/2}$ with $K > 0$.

Thus our final bound on the estimation error can be found by a global supremum approach and combining the above results with Lemma \ref{lemma-estimation-one} and Lemma \ref{lemma-estimation-two}. We start with the basic inequality:
\begin{align*}
R(\hat{d}; \theta) 
    &- R(\hat{d}^*; \theta) \\
    &= E\left\{ R_n(\hat{d}; \theta) - R_n(\hat{d}^*; \theta) \right\} \\
    &= E\left\{ R_n(\hat{d}; \theta) - \hat{R}_1(\hat{d}) + \hat{R}_1(\hat{d}) - \hat{R}_1(\hat{d}^*) + \hat{R}_1(\hat{d}^*) - R_n(\hat{d}^*; \theta) \right\} \\
    &\leq E\left\{ R_n(\hat{d}; \theta) - \hat{R}_1(\hat{d}) + \hat{R}_1(\hat{d}^*) - R_n(\hat{d}^*; \theta) \right\} \\
    &\leq 2 E \sup_{d \in \mathcal{D}} \left\vert R_n(\hat{d}; \theta) - \hat{R}_1(\hat{d}) \right\vert \\
    &= 2 E \sup_{d \in \mathcal{D}} \left\vert \frac{1}{n} \sum_{i=1}^n \{ \theta_i - d(X_i) \}^2 - \frac{1}{n} \sum_{i=1}^n \{ X_i - d(X_i) \}^2 - 2 \sigma^2 \frac{1}{n} \sum_{i=1}^n d'(X_i) + \sigma^2 \right\vert \\
    &= 2 E \sup_{d \in \mathcal{D}} \left\vert \frac{1}{n} \sum_{i=1}^n \{ \sigma^2 - (X_i - \theta_i)^2 - 2 \theta_i(X_i - \theta_i) + 2 (X_i - \theta_i)d(X_i) \right. \\
        &\qquad \left. \vphantom{\sum_{i=1}^n} - 2 \sigma^2 E_{X_i} d'(X_i) + 2 \sigma^2 E_{X_i} d'(X_i) - 2 \sigma^2 d'(X_i) \} \right\vert \\
    &\overset{*}{=} 2 E \sup_{d \in \mathcal{D}} \left\vert \frac{1}{n} \sum_{i=1}^n \{ \sigma^2 - (X_i - \theta_i)^2 - 2 \theta_i(X_i - \theta_i) + 2 (X_i - \theta_i)d(X_i) \right. \\
        &\qquad \left. \vphantom{\sum_{i=1}^n} - 2 E_{X_i} \{ (X_i - \theta_i) d(X_i) \} + 2 \sigma^2 E_{X_i} d'(X_i) - 2 \sigma^2 d'(X_i) \} \right\vert \\
    &\leq 2 E \left\vert \frac{1}{n}\sum_{i=1}^n (X_i - \theta_i)^2 - \sigma^2 \right\vert 
        + 4 E \left\vert \frac{1}{n} \sum_{i=1}^n (X_i-\theta_i) \theta_i \right\vert \\
        &\qquad + 4 E \sup_{d\in\mathcal{D}_n} \left\vert \frac{1}{n}\sum_{i=1}^n \left\{(X_i - \theta_i) d(X_i) - E_{X_i} (X_i - \theta_i) d(X_i) \right\} \right\vert \\
        &\qquad + 4 E \sup_{d\in\mathcal{D}_n} \left\vert \frac{1}{n}\sum_{i=1}^n \left\{d'(X_i) - E_{X_i} d'(X_i) \right\} \right\vert \\
    &\leq \frac{2 \sqrt{2} \sigma^2}{\sqrt{n}} + \frac{2 \sigma C_n}{\sqrt{n}} + \mathcal{O}\left[ n^{-1/2} \left\{ \log n + C_n (\log n)^{1/2} \right\} \right] \\
        &\qquad+ \mathcal{O}\left[ n^{-1/2} \left\{ C_n + \sigma (\log n)^{1/2} \right\}^2 \right] ,
\end{align*}
where the * equality is an application of Stein's lemma \cite{SteinsLemma}. 

Finally, we recall that $b_n = \sigma (K \log n)^{1/2}$ is hidden in the numerator of the final two big-O terms above as a multiplicative factor. This suggests that we take $K = 3$ so that both the estimation error and approximation error go to zero at the same rate, up to log-terms. Thus 
$$
R(\hat{d}; \theta) - R(d^*; \theta) = \mathcal{O}\left\{ n^{-1/2} \left(C_n + \log^{1/2} n \right)^2 \right\} .
$$
\end{proof}

\begin{prop}
Let $Z \sim N(\mu,\sigma^2)$. Let $g:\mathbb{R}\to\mathbb{R}$ have finite total variation and let $\mathcal{J}(g) = \{t_1, t_2, \dots \}$ be the countable set of locations where $g$ has a discontinuity. Let $g'$ be the derivative of $g$ almost everywhere and assume that $E\vert g'(Z) \vert < \infty$. Then
$$
\frac{1}{\sigma^2} E[(Z-\mu) g(Z-\mu)] =  E[g'(Z-\mu)] + \sum_{t_k \in \mathcal{J}(g)} \phi_\sigma(t_k-\mu) \left\{ \lim_{x \downarrow t_k} g(x) - \lim_{x \uparrow t_k} g(x) \right\} .
$$
\end{prop}

\begin{proof}
Let $Z \sim N(0,1)$ let $\phi$ denote the density function of $Z$. Fix $f : \mathbb{R} \to \mathbb{R}$ such that $f$ has bounded variation, $TV(f) < \infty$. Because $f$ has bound variation, it has at most a countable number of discontinuities (this follows from $f$ being the difference of two bounded, monotone non-decreasing function), which must be either jump-type or removable-type discontinuities. Let $\mathcal{J}(f) = \{t_1, t_2, \dots \}$ denote the set locations where $f$ has a discontinuity. Assume the following notational shorthand:
$$
\lim_{t \uparrow x} f(t) = f(x)_- \quad\text{and}\quad \lim_{t \downarrow x} f(t) = f(x)_+ .
$$

Observe that
\begin{align*}
E\{f'(Z)\}
    &= \int_{-\infty}^{\infty} f'(z) \phi(z) dz \\
    &= \int_{0}^{\infty} f'(z) \phi(z) dz - \int_{-\infty}^{0} f'(z) \{-\phi(z)\} dz \\
    &\overset{(i)}{=} \int_{0}^{\infty} f'(z) \left\{ \int_x^\infty t \phi(t) dt \right\} dz - \int_{-\infty}^{0} f'(z) \left\{ \int_{-\infty}^z t \phi(t) dt \right\} dz \\
    &\overset{(ii)}{=} \int_{0}^{\infty} t \phi(t) \left\{ \int_0^t f'(z) dz \right\} dt - \int_{-\infty}^{0} t \phi(t) \left\{ \int_t^0 f'(z) dz \right\} dt ,
\end{align*}
where $\phi(t)$ is the density function of a $N(0,1)$ random variable, (i) comes from $-t \phi(t) = \phi'(t)$, and (ii) is Fubini’s Theorem. Notice that by the fundamental theorem of Lebesgue integral calculus we have:
$$
f(t)_+ - f(0)_- = \int_0^t f'(x) dx + \sum_{t_i \in \mathcal{J}(f)} \{ f(t_i)_+ - f(t_i)_- \} 1(0 \leq t_i \leq t) .
$$
Now we can rearrange the first term in the above decomposition of $E\{f'(Z)\}$:
\begin{align*}
 \int_{0}^{\infty} t \phi(t) & \left\{ \int_0^t f'(z) dz \right\} dt \\
    &= \int_{0}^{\infty} t \phi(t) \left[ f(t)_+ - f(0)_- - \sum_{t_i \in \mathcal{J}(f)} \{ f(t_i)_+ - f(t_i)_- \} 1(0 \leq t_i \leq t) \right] dt \\
    &= \int_{0}^{\infty} t \phi(t) \left\{ f(t)_+ - f(0)_- \right\} dt \\
        &\qquad - \int_0^{\infty} t \phi(t) \left[ \sum_{t_i \in \mathcal{J}(f)} \{ f(t_i)_+ - f(t_i)_- \} 1(0 \leq t_i \leq t) \right] dt \\
    &\overset{(i)}{=} \int_{0}^{\infty} t \phi(t) \left\{ f(t)_+ - f(0)_- \right\} dt \\
        &\qquad - \sum_{t_i \in \mathcal{J}(f)} \{ f(t_i)_+ - f(t_i)_- \} \int_0^{\infty} t \phi(t) 1(0 \leq t_i \leq t) dt \\
    &= \int_{0}^{\infty} t \phi(t) \left\{ f(t)_+ - f(0)_- \right\} dt - \sum_{t_i \in \mathcal{J}(f) : t_i > 0} \{ f(t_i)_+ - f(t_i)_- \} \int_{t_i}^{\infty} t \phi(t) dt \\
    &= \int_{0}^{\infty} t \phi(t) \left\{ f(t)_+ - f(0)_- \right\} dt \\
        &\qquad- \sum_{t_i \in \mathcal{J}(f) : t_i > 0} \{ f(t_i)_+ - f(t_i)_- \} E\{Z \cdot 1(Z \geq t_i) \} ,
\end{align*}
where (i) is Fubini's theorem, noting that $\mathcal{J}(f)$ may be countably infinite. The second term follows analogously. Notice that $E \{Z \cdot 1(Z \leq a) \} = -\phi(a)$ and $E\{ Z \cdot 1(Z \geq a) \} = \phi(a)$ for all $a \geq 0$. So, putting these results together we get:
\begin{align*}
E&\{f'(Z)\} \\
    &= \left[ \int_{0}^{\infty} t \phi(t) \left\{ f(t)_+ - f(0)_- \right\} dt - \sum_{t_i \in \mathcal{J}(f) : t_i > 0} \{ f(t_i)_+ - f(t_i)_- \} E\{Z \cdot 1(Z \geq t_i) \} \right] \\
        & ~~ - \left[ \int_{-\infty}^{0} t \phi(t) \left\{ f(0)_+ - f(t)_- \right\} dt - \sum_{t_i \in \mathcal{J}(f) : t_i < 0} \{ f(t_i)_+ - f(t_i)_- \} E\{Z \cdot 1(Z \leq t_i) \} \right] \\
    &= E\{ Z f(Z)\} - \{ f(0)_+ - f(0)_- \} \phi(0) \\
        &\quad- \sum_{t_i \in \mathcal{J}(f) : t_i > 0} \{ f(t_i)_+ - f(t_i)_- \} \phi(t_i) - \sum_{t_i \in \mathcal{J}(f) : t_i < 0} \{ f(t_i)_+ - f(t_i)_- \} \phi(t_i)  \\
    &= E\{ Z f(Z)\} - \sum_{t_i \in \mathcal{J}(f) \cup \{0\}} \{ f(t_i)_+ - f(t_i)_- \} \phi(t_i) \\
    &= E\{ Z f(Z)\} - \sum_{t_i \in \mathcal{J}(f)} \{ f(t_i)_+ - f(t_i)_- \} \phi(t_i) .
\end{align*}
The last equality comes from the fact that if $0 \not\in \mathcal{J}(f)$, then $f(0)_+ - f(0)_- = 0$. We note that by our use of the fundamental theorem of Lebesgue integral calculus, removable-type discontinuities do not affect the value of the expectation because $f(t_i)_+ - f(t_i)_- = 0$ when $f$ has a removable-type discontinuity at $t_i$. We chose the notation $\mathcal{J}(f)$ to emphasize the importance of only jump-type discontinuities based on this fact.

Finally our result follows by a change of variables. Let $\mu, \sigma \in \mathbb{R}$ such that $\sigma > 0$. Define $Z = (X - \mu)/\sigma$ and $f(z) = g(\sigma z + \mu)$. Now,
\begin{align*}
\sigma^{-2} E\{ (X - \mu) h(X) \}
    &= \sigma^{-1} E\{ Z f(Z) \} \\
    &= \sigma^{-1} \left[ E\{ f'(Z)\} + \sum_{t_i \in \mathcal{J}(f)} \{ f(t_i)_+ - f(t_i)_- \} \phi(t_i) \right] \\
    &= \sigma^{-1} \left[ E\{ \sigma g'(X)\} + \sum_{u_i \in \mathcal{J}(g)} \{ g(u_i)_+ - g(u_i)_- \} \phi(\sigma^{-1}(u_i - \mu)) \right] \\
    &= E\{ g'(X)\} + \sum_{u_i \in \mathcal{J}(g)} \{ g(u_i)_+ - g(u_i)_- \} \phi_\sigma(u_i - \mu) ,
\end{align*}
where $\phi_\sigma(t)$ is the density function of a $N(0,\sigma^2)$ random variable.
\end{proof}

\begin{thrm}
Let $h > 0$. For every $g\in\mathcal{D}_{0,n}$ \eqref{eq-D0} that is right continuous at $X_1, \dots, X_n$, there exists a piecewise constant function, $\tilde{g}\in\mathcal{D}_{0,n}$, that has at most $n-1$ knots and satisfies $\tilde{g}(X_i) = g(X_i)$ from the right and $\hat{R}_0(\tilde{g}; h) \leq \hat{R}_0(g; h)$. The knots of $\tilde{g}(x)$ lie at the minimum of $\hat{f}_h(t)$ between each consecutive order statistics of $X_1, \dots, X_n$.
\end{thrm}
\begin{proof}
Let $x_1 < \dots < x_n$ be the order statistics of $X_1, \dots, X_n$ and let $g\in\mathcal{D}_{0,n}$ be right continuous at each $x_i$. First notice that by the extreme value theorem, because $\hat{f}_h(t)$ is continuous, there is a minimum that it obtains at least once on any closed interval $[a,b] \subset \mathbb{R}$, so
\begin{align*}
\int_a^b \hat{f}_h(x) g'(x) dx
    &\geq \int_a^b \left\{ \min_{x\in[a,b]} \hat{f}_h(x) \right\} g'(x) dx \\
    &= \left\{ \min_{x\in[a,b]} \hat{f}_h(x) \right\} \int_a^b g'(x) dx . \\
\end{align*}

Now, let $\tilde{g} \in \mathcal{D}_{0,n}$ be a degree-zero (piecewise constant) spline such that $\tilde{g}(x_i) = g(x_i)$ for each $i = 1, \dots, n$ from the right. This $\tilde{g} \in \mathcal{D}_{0,n}$ exists because its piecewise constant structure ensures that $TV(\tilde{g}) \leq TV(g)$ and if $g$ is bounded or monotone, $\tilde{g}$ is has the same bounds and is monotone. Let $t_i = \arg\min_{x\in[x_{i}, x_{i+1}]} \hat{f}_h(x)$ for $i = 1, \dots, n-1$ be the $n-1$ knots for $\tilde{g}$. Notice that because $x_i$ are distinct by assumption (and $X_i$ are distinct almost surely), the definition of $\hat{f}_h(t)$ \eqref{eq-kde} gives us that each of the $t_i$ must be distinct and lie in the open interval $(x_i,x_{i+1})$. Further, notice that $\tilde{g}$ has at most $n-1$ jump discontinuities and is constant on the resulting $n$ open intervals: $(-\infty, t_1), (t_1, t_2), \dots, (t_n, \infty)$. 

From the fundamental theorem of Lebesgue integral calculus we have:
$$
g(b)_+ - g(a)_- = \int_a^b g'(x) dx + \sum_{t_k \in \mathcal{J}(g) \cap [a,b]} \{ g(t_k)_+ - g(t_k)_- \} .
$$
Combining these two results and using the definition of $t_i$, we get that:
\begin{align*}
\int_{x_i}^{x_{i+1}} \hat{f}_h(x) g'(x) dx +& \sum_{t_k \in \mathcal{J}(g) \cap [x_i, x_{i+1}]} \hat{f}_h(t_k) \{ g(t_k)_+ - g(t_k)_- \} \\
    &\geq \hat{f}_h(t_i)  \int_{x_i}^{x_{i+1}} g'(x) dx + \hat{f}_h(t_i) \sum_{t_k \in \mathcal{J}(g) \cap [x_i, x_{i+1}]} \{ g(t_k)_+ - g(t_k)_- \} \\
    &= \hat{f}_h(t_i) \left[ \int_{x_i}^{x_{i+1}} g'(x) dx + \sum_{t_k \in \mathcal{J}(g) \cap [x_i, x_{i+1}]} \{ g(t_k)_+ - g(t_k)_- \} \right] \\
    &= \hat{f}_h(t_i) \{ g(x_{i+1})_+ - g(x_i)_- \} .
\end{align*}
Also, notice that because $g$ is monotone non-decreasing (one approach to extending this proof to functions with bound variation is to restrict the domain to $[x_1, x_n]$, to avoid potentially difficult asymptotic behavior),
\begin{align*}
\int_{-\infty}^{\infty} \hat{f}_h(x) g'(x) dx 
    &= \int_{-\infty}^{x_1} \hat{f}_h(x) g'(x) dx + \int_{x_1}^{x_{n}} \hat{f}_h(x) g'(x) dx + \int_{x_n}^{\infty} \hat{f}_h(x) g'(x) dx \\
    &\geq 0 + \int_{x_1}^{x_{n}} \hat{f}_h(x) g'(x) dx + 0 \\
    &= \sum_{i=1}^{n-1} \int_{x_i}^{x_{i+1}} \hat{f}_h(x) g'(x) dx ,
\end{align*}
and
\begin{align*}
\sum_{t_k \in \mathcal{J}(g)} \hat{f}_h(t_k) \{g(t_k)_+ - g(t_k)_-\}
    &= \sum_{t_k \in \mathcal{J}(g) \cap (-\infty,x_1)} \hat{f}_h(t_k) \{g(t_k)_+ - g(t_k)_-\} \\
        &\quad\quad+ \sum_{t_k \in \mathcal{J}(g) \cap [x_1,x_{n})} \hat{f}_h(t_k) \{g(t_k)_+ - g(t_k)_-\} \\
        &\quad\quad+ \sum_{t_k \in \mathcal{J}(g) \cap [x_n,\infty)} \hat{f}_h(t_k) \{g(t_k)_+ - g(t_k)_-\} \\
    &\geq 0 + \sum_{t_k \in \mathcal{J}(g) \cap [x_1,x_{n})} \hat{f}_h(t_k) \{g(t_k)_+ - g(t_k)_-\} + 0 \\
    &= \sum_{i=1}^{n-1} \sum_{t_k \in \mathcal{J}(g) \cap [x_i,x_{i+1})} \hat{f}_h(t_k) \{g(t_k)_+ - g(t_k)_-\} .
\end{align*}

Finally, because $\tilde{g}(x_i) = g(x_i)$ from the right (and $g$ is right continuous) for each $i = 1, \dots, n$ and $\tilde{g}$ has its jump points at $t_i \in (x_i, x_{i+1})$,
\begin{align*}
\hat{R}_0(g; h) 
    &= \frac{1}{n} \sum_{i=1}^n (x_i - g(x_i))^2 + 2 \sigma^2 \int_{-\infty}^{\infty} \hat{f}_h(x) g'(x) dx \\
        &\qquad + 2 \sigma^2 \sum_{t_k \in \mathcal{J}(g)} \hat{f}_h(t_k) \{g(t_k)_+ - g(t_k)_-\}  - \sigma^2 \\
    &\geq \frac{1}{n} \sum_{i=1}^n (x_i - g(x_i))^2 + 2 \sigma^2 \sum_{i=1}^{n-1} \left[ \int_{x_i}^{x_{i+1}} \hat{f}_h(x) g'(x) dx \right. \\
        &\qquad \left. \vphantom{\int_{x_i}^{x_{i+1}}} + \sum_{t_k \in \mathcal{J}(g) \cap [x_i,x_{i+1})} \hat{f}_h(t_k) \{g(t_k)_+ - g(t_k)_-\} \right] - \sigma^2 \\
    &\geq \frac{1}{n} \sum_{i=1}^n (x_i - g(x_i))^2 - \sigma^2 \\
        &\qquad + 2 \sigma^2 \sum_{i=1}^{n-1} \hat{f}_h(t_i) \left[ \int_{x_i}^{x_{i+1}} g'(x) dx + \sum_{t_k \in \mathcal{J}(g) \cap [x_i,x_{i+1})} \{g(t_k)_+ - g(t_k)_-\} \right] \\
    &= \frac{1}{n} \sum_{i=1}^n (x_i - g(x_i))^2 + 2 \sigma^2 \sum_{i=1}^{n-1} \hat{f}_h(t_i) \{g(x_{i+1})_- - g(x_i)_+\} - \sigma^2 \\
    &\overset{(i)}{=} \frac{1}{n} \sum_{i=1}^n (x_i - \tilde{g}(x_i))^2 + 2 \sigma^2 \sum_{i=1}^{n-1} \hat{f}_h(t_i) \{\tilde{g}(x_{i+1}) - \tilde{g}(x_i)\} - \sigma^2 \\
    &= \frac{1}{n} \sum_{i=1}^n (x_i - \tilde{g}(x_i))^2 + 2 \sigma^2 \sum_{t_k \in \mathcal{J}(\tilde{g})} \hat{f}_h(t_k) \{\tilde{g}(t_k)_+ - \tilde{g}(t_k)_-\} - \sigma^2 \\
    &\overset{(ii)}{=}  \hat{R}_0(\tilde{g}; h) .
\end{align*}
W note that $\tilde{g}$ is constant almost everywhere and only has discontinuities at $t_i \in (x_i, x_{i+1})$. These facts imply: (i) $g(x_{i+1})_- = g(x_{i+1})_+ = g(x_{i+1})$ for each $i=1, \dots, n-1$, and (ii) the integral term in $\hat{R}_0(\tilde{g}; h)$ is zero.
\end{proof}

\begin{thrm}
Let $X_i \sim N(\theta_i, \sigma^2)$ be independent such that $\sigma^2$ is known and $\max_i \vert \theta_i \vert < C_n$. Further, let $b_n = \{(8/3) \sigma^2 \log(n) \}^{1/2}$ and $h_n > 0$ such that $h_n \asymp \sigma n^{-1/6}$. Define $\tilde{d}_{h_n}(x)$ as in \eqref{eq-estimator0} and let $d^*(x)$ be the optimal separable estimator \eqref{oracle}, then
$$
R(\tilde{d}_{h_n}; \theta) - R(d^*; \theta) = \mathcal{O}\left\{ n^{-1/3} \left(C_n + \log^{1/2} n \right)^2 \right\} .
$$
\end{thrm}
\begin{proof}
Most of proof follows identically to the proof of Theorem \ref{thrm-pwl-rate} except additional care needs to be paid to the estimation error of the penalty and the approximation error's probability bound has a 4 rather than a 6 because the function class no longer has its total variation constraint. 

Define $f_v(x) = n^{-1} \sum_{i=1}^n \phi_v(x-\theta_i)$ for $v = (\sigma^2 + h_n^2)^{1/2}$, observe that $f_v(x) = E \hat{f}_{h_n}(x)$ and $f_v'(x) = E \hat{f}_{h_n}'(x)$. For brevity of notation, define $\beta_d(t_k) = d(t_k)_+ - d(t_k)_-$ for any $d\in\mathcal{D}_{n,0}$ and corresponding $\mathcal{J}(d)$. Rather than the penalty estimation term in the proof of Theorem \ref{thrm-pwl-rate}:
$$
E \sup_{d\in\mathcal{D}_{0,n}} \left\vert \frac{1}{n}\sum_{i=1}^n \left\{ d'(X_i) - E_{X_i} d'(X_i) \right\} \right\vert ,
$$
our task is now to bound: 
\begin{align*}
E \sup_{d\in\mathcal{D}_{0,n}} \left\vert \int_{-\infty}^{\infty} f_\theta(x) d'(x) dx \right. &+ \sum_{t_k\in\mathcal{J}(d)} \beta_d(t_k) f_\theta(t_k) \\
    \qquad\qquad & \left. - \int_{-\infty}^{\infty} \hat{f}_{h_n}(x) d'(x) dx - \sum_{t_k\in\mathcal{J}(d)} \beta_d(t_k) \hat{f}_{h_n}(t_k) \right\vert .
\end{align*}
This term is challenging to analyze because both $\mathcal{D}_{0,n}$ and $\hat{f}_{h_n}(x)$ are random. To address this challenge, we introduce $f_v(x)$ and use the triangle inequality to study both a structural error term:
\begin{align*}
E \sup_{d\in\mathcal{D}_{0,n}} \left\vert \int_{-\infty}^{\infty} f_\theta(x) d'(x) dx \right. &+ \sum_{t_k\in\mathcal{J}(d)} \beta_d(t_k) f_\theta(t_k) \\
    \qquad\qquad & \left. - \int_{-\infty}^{\infty} f_v(t_k) d'(x) dx - \sum_{t_k\in\mathcal{J}(d)} \beta_d(t_k) f_v(t_k) \right\vert
\end{align*}
and a random error term:
\begin{align*}
E \sup_{d\in\mathcal{D}_{0,n}} \left\vert \int_{-\infty}^{\infty} f_v(x) d'(x) dx \right. &+ \sum_{t_k\in\mathcal{J}(d)} \beta_d(t_k) f_v(t_k) \\
    \qquad\qquad & \left. - \int_{-\infty}^{\infty} \hat{f}_{h_n}(x) d'(x) dx - \sum_{t_k\in\mathcal{J}(d)} \beta_d(t_k) \hat{f}_{h_n}(t_k) \right\vert .
\end{align*}
 
Starting with the structural term, we apply Proposition \ref{prop-extendedSL} utilizing the fact that $f_v(t) = E \hat{f}_{h_n}(x)$ and that $v \geq \sigma > 0$. Define $Y_i = X_i + Z_i$ where $Z_i \sim N(0, h_n^2)$ is independent from all other random variables. Now we can bound the structural error utilizing the fact that the only randomness in this term is in $\mathcal{D}_{0,n}$:
\begin{align*}
E \sup_{d\in\mathcal{D}_{0,n}} & \left\vert \int_{-\infty}^{\infty} f_\theta(x) d'(x) dx + \sum_{t_k\in\mathcal{J}(d)} \beta_d(t_k) f_\theta(t_k) \right. \\
  & \qquad\qquad \left. - \int_{-\infty}^{\infty} f_v(x) d'(x) dx - \sum_{t_k\in\mathcal{J}(d)} \beta_d(t_k) f_v(t_k) \right\vert \\
&= E \sup_{d\in\mathcal{D}_{0,n}} \left\vert \frac{1}{n} \sum_{i=1}^n E_{X_i} \left\{\left(\frac{X_i - \theta_i}{\sigma^2}\right) d(X_i)\right\} - \frac{1}{n} \sum_{i=1}^n E_{Y_i} \left\{ \left(\frac{Y_i - \theta_i}{v^2}\right) d(Y_i) \right\} \right\vert \\
&= E \sup_{d\in\mathcal{D}_{0,n}} \left\vert \frac{1}{n} \sum_{i=1}^n \int_{-\infty}^{\infty} \left(\frac{x - \theta_i}{\sigma^2}\right) \frac{1}{\sigma} \phi\left(\frac{x - \theta_i}{\sigma}\right) d(x) dx \right. \\
  &\qquad\qquad \left.- \frac{1}{n} \sum_{i=1}^n \int_{-\infty}^{\infty} \left(\frac{x - \theta_i}{v^2}\right) \frac{1}{v} \phi\left(\frac{x - \theta_i}{v}\right) d(x) dx \right\vert \\
&\leq E \sup_{d\in\mathcal{D}_{0,n}} \frac{1}{n} \sum_{i=1}^n \left\vert \int_{-\infty}^{\infty} \left(\frac{x - \theta_i}{\sigma^2}\right) \frac{1}{\sigma} \phi\left(\frac{x - \theta_i}{\sigma}\right) d(x) dx \right. \\
  &\qquad\qquad \left.- \int_{-\infty}^{\infty} \left(\frac{x - \theta_i}{v^2}\right) \frac{1}{v} \phi\left(\frac{x - \theta_i}{v}\right) d(x) dx \right\vert \\
&= E \sup_{d\in\mathcal{D}_{0,n}} \frac{1}{n} \sum_{i=1}^n \left\vert \int_{-\infty}^{\infty} \left\{ \frac{y}{\sigma^3} \phi\left(\frac{y}{\sigma}\right) - \frac{y}{v^3} \phi\left(\frac{y}{v}\right) \right\} d(y+\theta_i) dy \right\vert \\
&\leq E \left\{ \sup_{d\in\mathcal{D}_{0,n}} \sup_{x\in\mathbb{R}} \vert d(x) \vert \right\} \cdot  \int_{-\infty}^{\infty} \left\vert \frac{y}{\sigma^3} \phi\left(\frac{y}{\sigma}\right) - \frac{y}{v^3} \phi\left(\frac{y}{v}\right) \right\vert dy \\
&= \frac{2}{(2 \pi)^{1/2}} \left(\frac{1}{\sigma} - \frac{1}{v}\right) \cdot E \left\{ \sup_{d\in\mathcal{D}_{0,n}} \sup_{x\in\mathbb{R}} \vert d(x) \vert \right\} \\
&\leq \frac{2}{(2 \pi)^{1/2}} \left(\frac{1}{\sigma} - \frac{1}{v}\right) \cdot E \left\{ \max_{i=1}^n \vert X_i\vert + b_n \right\} \\
&\leq \frac{2}{(2 \pi)^{1/2}} \left(\frac{1}{\sigma} - \frac{1}{v}\right) \cdot E \left\{ \max_{i=1}^n \vert X_i - \theta_i\vert + C_n + b_n \right\} \\
&\leq \frac{2}{(2 \pi)^{1/2}} \left(\frac{1}{\sigma} - \frac{1}{v}\right) \cdot \left[ \{2 \sigma^2 \log(2n) \}^{1/2} + C_n + b_n \right] \\
&= \frac{2 \left[ \{2 \sigma^2 \log(2n) \}^{1/2} + C_n + b_n \right]}{(2 \pi)^{1/2}} \left\{ \sigma^{-1} - (\sigma^2 + h_n^2)^{-1/2} \right\} . \\
\end{align*}
The aggressive use of triangle inequality in the first inequality may seem overly aggressive at first, but we note that if $\theta_1 = \dots = \theta_n = 0$, then the bound is actually an equality. The following step where we sup-out the $d(x)$ may be loose. The last inequality is an application of a Gaussian maximal inequality \cite{BoucheronLugosiMassart2013}. We observe that as $h_n \downarrow 0$, the structural error term converges to zero. 

Next we bound the random error term. To do this we will again use symmetrization and a chaining bound utilizing the fact that $f_v(t) = E \hat{f}_{h_n}(t)$; however, we will first need to define and study the new function class:
$$
\mathcal{G}_{n,h} = \left\{ g(y) = \int_{-\infty}^{\infty} \phi_h(x-y) d'(x) dx + \sum_{t_k\in\mathcal{J}(d)} \phi_h(t_k-y) \beta_d(t_k) ~ : ~ d\in\mathcal{D}_{0,n} \right\}
$$
for some $h > 0$. In particular, we will demonstrate that $g \in \mathcal{G}_{n,h}$ is bound and has bound total variation. Fixed $h > 0$ and let $\tau_n$ be such that $TV(d) \leq \tau_n < \infty$ for every $d\in\mathcal{D}_{0,n}$ (for monotone functions this is just the range). We begin by showing that $\mathcal{G}_{n,h}$ is uniformly bounded. Observe that by the triangle inequality, we can bound the range of every $g(y)$ as follows:
\begin{align*}
\left\vert g(y) \right\vert
    &= \left\vert \int_{-\infty}^{\infty} \phi_h(x-y) d'(x) dx + \sum_{t_k\in\mathcal{J}(d)} \phi_h(t_k-y) \beta_d(t_k) \right\vert \\
    &\leq \int_{-\infty}^{\infty} \phi_h(x-y) \vert d'(x) \vert dx + \sum_{t_k\in\mathcal{J}(d)} \phi_h(t_k-y) \vert \beta_d(t_k) \vert \\
    &\leq (2 \pi h^2)^{-1/2} \int_{-\infty}^{\infty} \vert d'(x) \vert dx + (2 \pi h^2)^{-1/2} \sum_{t_k\in\mathcal{J}(d)} \vert \beta_d(t_k) \vert \\
    &= (2 \pi h^2)^{-1/2} \cdot TV(d) \\
    &\leq \tau_n (2 \pi h^2)^{-1/2} .
\end{align*}
Next we bound the total variation of $g(y)$ by applying the triangle inequality and Fubini's theorem:
\begin{align*}
TV(g) 
    &= \int_{-\infty}^{\infty} \left\vert g'(y) \right\vert dy \\
    &= \int_{-\infty}^{\infty} \left\vert \frac{\partial}{\partial y} \int_{-\infty}^{\infty} \phi_h(x-y) d'(x) dx + \frac{\partial}{\partial y} \sum_{t_k\in\mathcal{J}(d)} \phi_h(t_k-y) \beta_d(t_k) \right\vert dy \\
    &= \int_{-\infty}^{\infty} \left\vert \int_{-\infty}^{\infty} \frac{\partial}{\partial y} \{ \phi_h(x-y) d'(x) \} dx + \sum_{t_k\in\mathcal{J}(d)} \frac{\partial}{\partial y} \{ \phi_h(t_k-y) \beta_d(t_k) \} \right\vert dy \\
    &= \int_{-\infty}^{\infty} \left\vert \int_{-\infty}^{\infty} \frac{x-y}{h^2} \phi_h(x-y) d'(x) dx + \sum_{t_k\in\mathcal{J}(d)} \frac{t_k-y}{h^2} \phi_h(t_k-y) \beta_d(t_k) \right\vert dy \\
    &\leq \int_{-\infty}^{\infty} \int_{-\infty}^{\infty} \frac{\vert x-y\vert}{h^2} \phi_h(x-y) \vert d'(x)\vert dx ~ dy \\
        &\qquad+ \int_{-\infty}^{\infty} \sum_{t_k\in\mathcal{J}(d)} \frac{\vert t_k-y\vert}{h^2} \phi_h(t_k-y) \vert\beta_d(t_k)\vert dy \\
    &= \int_{-\infty}^{\infty} \int_{-\infty}^{\infty} \frac{\vert x-y\vert}{h^2} \phi_h(x-y) \vert d'(x)\vert dy ~ dx \\
        &\qquad+ \sum_{t_k\in\mathcal{J}(d)} \int_{-\infty}^{\infty} \frac{\vert t_k-y\vert}{h^2} \phi_h(t_k-y) \vert\beta_d(t_k)\vert dy \\
    &= \int_{-\infty}^{\infty} 2 (2 \pi h^2)^{-1/2} \vert d'(x)\vert dx + \sum_{t_k\in\mathcal{J}(d)} 2 (2 \pi h^2)^{-1/2} \vert\beta_d(t_k)\vert \\
    &= 2 (2 \pi h^2)^{-1/2} \cdot TV(d) \\
    &\leq 2 \tau_n (2 \pi h^2)^{-1/2} .
\end{align*}
Thus $\mathcal{G}_{n,h}$ inherits the boundedness and bounded total variation from $\mathcal{D}_{0,n}$ for any $h > 0$. Moreover, both the range and total variation of functions in $\mathcal{G}_{n,h}$ are bound by $2 \tau_n (2 \pi h^2)^{-1/2}$, which will make our chaining bounds simple. 

Before proceeding with the symmetrization argument, observe that by Fubini's theorem and our definition of $\mathcal{G}_{n,h}$, for every $d\in\mathcal{D}_{0,n}$ and $h_n > 0$ there exists a $g \in \mathcal{G}_{n,h_n}$ such that:
\begin{align*}
\int_{-\infty}^{\infty} & f_v(x) d'(x) dx + \sum_{t_k\in\mathcal{J}(d)} \beta_d(t_k) f_v(t_k) \\
    &= \frac{1}{n} \sum_{i=1}^n \int_{-\infty}^{\infty} \phi_v(x-\theta_i) d'(x) dx + \frac{1}{n} \sum_{i=1}^n \sum_{t_k\in\mathcal{J}(d)} \beta_d(t_k) \phi_v(x-\theta_i) \\
    &= \frac{1}{n} \sum_{i=1}^n \int_{-\infty}^{\infty} E_{X_i} \phi_h(x-X_i) d'(x) dx + \frac{1}{n} \sum_{i=1}^n \sum_{t_k\in\mathcal{J}(d)} \beta_d(t_k) E_{X_i} \phi_h(x-X_i) \\
    &= \frac{1}{n} \sum_{i=1}^n E_{X_i} \left\{ \int_{-\infty}^{\infty} \phi_h(x-X_i) d'(x) dx + \sum_{t_k\in\mathcal{J}(d)} \beta_d(t_k) \phi_h(x-X_i) \right\} \\
    &= \frac{1}{n} \sum_{i=1}^n E_{X_i} g(X_i) \\
\end{align*}
and
\begin{align*}
\int_{-\infty}^{\infty} & \hat{f}_{h_n}(x) d'(x) dx + \sum_{t_k\in\mathcal{J}(d)} \beta_d(t_k) \hat{f}_{h_n}(t_k) \\
    &= \frac{1}{n} \sum_{i=1}^n \int_{-\infty}^{\infty} \phi_{h_n}(x-X_i) d'(x) dx + \frac{1}{n} \sum_{i=1}^n \sum_{t_k\in\mathcal{J}(d)} \beta_d(t_k) \phi_{h_n}(x-\theta_i) \\
    &= \frac{1}{n} \sum_{i=1}^n g(X_i) .
\end{align*}
Putting these together, we can proceed with our symmetrization argument. Let $X_i'$ be independent copies of $X_i$ and let $\epsilon_i$ be independent Rademacher random variables. Then,
\begin{align*}
E_X \sup_{d\in\mathcal{D}_n} & \left\vert \int_{-\infty}^{\infty} f_v(x) d'(x) dx + \sum_{t_k\in\mathcal{J}(d)} \beta_d(t_k) f_v(t_k) \right. \\
        &\qquad\qquad \left. - \int_{-\infty}^{\infty} \hat{f}_{h_n}(x) d'(x) dx - \sum_{t_k\in\mathcal{J}(d)} \beta_d(t_k) \hat{f}_{h_n}(t_k) \right\vert \\
    &= E_X \sup_{g\in\mathcal{G}_{n,h_n}} \left\vert \frac{1}{n} \sum_{i=1}^n E_{X_i} g(X_i) - \frac{1}{n} \sum_{i=1}^n g(X_i) \right\vert \\
    &= E_X \sup_{g\in\mathcal{G}_{n,h_n}} \left\vert \frac{1}{n} \sum_{i=1}^n E_{X_i'} g(X_i') - \frac{1}{n} \sum_{i=1}^n g(X_i) \right\vert \\
    &= E_X \sup_{g\in\mathcal{G}_{n,h_n}} \left\vert \frac{1}{n} \sum_{i=1}^n E_{X_i' \vert X_i} \{ g(X_i') - g(X_i) \} \right\vert \\
    &= E_X \sup_{g\in\mathcal{G}_{n,h_n}} \left\vert E_{X' \vert X} \frac{1}{n} \sum_{i=1}^n \{ g(X_i') - g(X_i) \} \right\vert \\
    &\leq E_{X,X'} \sup_{g\in\mathcal{G}_{n,h_n}} \left\vert \frac{1}{n} \sum_{i=1}^n \{ g(X_i') - g(X_i) \} \right\vert \\
    &= E_{X,X',\epsilon} \sup_{g\in\mathcal{G}_{n,h_n}} \left\vert \frac{1}{n} \sum_{i=1}^n \epsilon_i \{ g(X_i') - g(X_i) \} \right\vert \\
    &\leq 2 E_{X,\epsilon} \sup_{g\in\mathcal{G}_{n,h_n}} \left\vert \frac{1}{n} \sum_{i=1}^n \epsilon_i g(X_i) \right\vert .
\end{align*}

Now we will use a chaining bound to finish our bound on the random term. Define $D_n = 2 \tau_n (2 \pi h_n^2)^{-1/2}$, where because $\mathcal{D}_{0,n}$ is uniformly bound and monotone non-decreasing, $\tau_n = \max_{i=1}^n X_{i} - \min_{i=1}^n X_{i} + 2 b_n$. Then, using our properties of $\mathcal{G}_{n,h_n}$, Lemma \ref{lemma-estimation-rademacher}, metric entropy results \cite{vandeGeer2000}, and Gaussian maximal inequalities \cite{BoucheronLugosiMassart2013} we have:
\begin{align*}
E_X \sup_{d\in\mathcal{D}_n} & \left\vert \int_{-\infty}^{\infty} f_v(x) d'(x) dx + \sum_{t_k\in\mathcal{J}(d)} \beta_d(t_k) f_v(t_k) \right. \\
        &\qquad \left. - \int_{-\infty}^{\infty} \hat{f}_{h_n}(x) d'(x) dx - \sum_{t_k\in\mathcal{J}(d)} \beta_d(t_k) \hat{f}_{h_n}(t_k) \right\vert \\
    &\leq 2 E_{X,\epsilon} \sup_{g\in\mathcal{G}_n} \left\vert \frac{1}{n} \sum_{i=1}^n \epsilon_i g(X_i) \right\vert \\
    &\leq 2 E_X \left[ C n^{-1/2} \int_0^{D_n} \{ \log N(\mathcal{G}_{n,h_n}; L_n; t) \}^{1/2} dt \right] \\
    &\leq 2 E_X \left\{ C n^{-1/2} \int_0^{D_n} (A_3 D_n / t) dt \right\} \\
    &= \frac{4 C \sqrt{A_3}}{\sqrt{n}} E_X \left\{ D_n \right\} \\
    &= 8 C A_3^{1/2} (2 \pi h_n^2 n)^{-1/2} E_X \left\{ \tau_n \right\} \\
    &= C_3 (2 \pi h_n^2 n)^{-1/2} E_X \left\{ \tau_n \right\} \\
    &= C_3 (2 \pi h_n^2 n)^{-1/2} E_X \left\{ X_{(n)} - X_{(1)} + 2 b_n \right\} \\
    &\leq C_3 (2 \pi h_n^2 n)^{-1/2} E_X \left\{ \max_{i=1}^n \vert X_i - \theta_i\vert + C_n + b_n \right\} \\
    &\leq C_3 (2 \pi h_n^2 n)^{-1/2} \left[ \{2 \sigma^2 \log(2 n)\}^{1/2} + C_n + b_n \right] ,
\end{align*}
where $A_3 > 0$ and $C_3 = 8 C A_3^{1/2}$ are a constants. We observe that as $h_n \downarrow 0$, the random error term grows. This suggests that we need to optimize $h_n$ to balance both the structural error and the random error.

Putting our results together we bound the estimation of the penalty term as:
\begin{align*}
E & \sup_{d\in\mathcal{D}_{0,n}} \left\vert \int_{-\infty}^{\infty} f_\theta(x) d'(x) dx + \sum_{t_k\in\mathcal{J}(d)} \beta_d(t_k) f_\theta(t_k) \right. \\
    & \qquad\qquad \left. - \int_{-\infty}^{\infty} \hat{f}_{h_n}(x) d'(x) dx - \sum_{t_k\in\mathcal{J}(d)} \beta_d(t_k) \hat{f}_{h_n}(t_k) \right\vert \\
    &\leq \frac{2 \left[ \{2 \sigma^2 \log(2n) \}^{1/2} + C_n + b_n \right]}{(2 \pi)^{1/2}} \left\{ \sigma^{-1} - (\sigma^2 + h_n^2)^{-1/2} \right\} \\
        &\qquad + C_3 (2 \pi h_n^2 n)^{-1/2} \left[ \{2 \sigma^2 \log(2 n)\}^{1/2} + C_n + b_n \right] \\
    &= \frac{C_3 \left[ \{2 \sigma^2 \log(2n) \}^{1/2} + C_n + b_n \right]}{(2 \pi)^{1/2}} \cdot \left[ (2 / C_3) \left\{ \sigma^{-1} - (\sigma^2 + h_n^2)^{-1/2} \right\} + (n h_n^2)^{-1/2} \right] \\
\end{align*}
This bound suggests that large $n$, the optimal $h_n$ is satisfies $h_n \asymp \sigma n^{-1/6}$. Thus we get our bound on the estimation error of the penalty term:
\begin{align*}
E \sup_{d\in\mathcal{D}_{0,n}} & \left\vert \int_{-\infty}^{\infty} f_\theta(x) d'(x) dx + \sum_{t_k\in\mathcal{J}(d)} \beta_d(t_k) f_\theta(t_k) \right. \\
    & \qquad\qquad \left. - \int_{-\infty}^{\infty} \hat{f}_{h_n}(x) d'(x) dx - \sum_{t_k\in\mathcal{J}(d)} \beta_d(t_k) \hat{f}_{h_n}(t_k) \right\vert \\
    &= \left[ \{2 \sigma^2 \log(2n) \}^{1/2} + C_n + b_n \right] \cdot \mathcal{O}\left( n^{-1/3} \right) \\
    &= \mathcal{O}\left\{ n^{-1/3} \left( C_n + \log^{1/2} n \right) \right\} ,
\end{align*}
when $b_n = \{ \sigma^2 K \log(n) \}^{1/2}$.

This rate bound dominates the convergence of the estimation error. With a slower rate of convergence in the estimation error, we choose a smaller value for $K$ to balance the convergence rates between the estimation error and approximation error. Letting $K = 8/3$, we get the following bound on approximation error:
$$
\mathcal{O}\left\{ n^{-1/3} \left(C_n + \log^{1/2} n \right)^2 \right\} .
$$
Putting these bounds together in the same manner as the proof of Theorem \ref{thrm-pwc-rate} completes the proof.
\end{proof}

\begin{corollary}
Let $X_i \sim N(\theta_i, \sigma^2)$ be independent such that $\sigma^2$ is known and $\max_{i=1}^n \vert \theta_i \vert < C_n$. Further, let $b_n = \{(8/3) \sigma^2 \log(n) \}^{1/2}$, $\lambda_n = \max_{i=1}^n X_i - \min_{i=1}^n X_i + 2 b_n$, and $h_n > 0$ such that $h_n \asymp \sigma n^{-1/6}$. Define the new function class: 
$$
\mathcal{C}_{0,n} = \left\{ d:\mathbb{R}\to\left[\min_i X_i - b_n, \max_i X_i + b_n\right] \mid TV(d) \leq \lambda_n \right\} , 
$$
and the corresponding estimator $\bar{d}_{h_n} = \arg\min_{d\in\mathcal{C}_{1,n}} \hat{R}_0(d; h_n)$. Let $d^*(x)$ be the optimal separable estimator \eqref{oracle}, then
$$
R(\bar{d}_{h_n}; \theta) - R(d^*; \theta) = \mathcal{O}\left\{ n^{-1/3} \left(C_n + \log^{1/2} n \right)^2 \right\} .
$$
\end{corollary}
\begin{proof}
This proof follows almost immediately from the proof of Theorem \ref{thrm-pwc-rate}. We notice that the addition of a total variation constraint increases the constant in the probability term bound in the approximation error from 4 to 6. Further, since the analysis of $\mathcal{G}_{n,h}$ is already in terms of total variation, it applies to the function class $C_{0,n}$ without modification.
\end{proof}

\end{appendix}


\bibliographystyle{abbrvnat}
\bibliography{paper}